\documentclass[letterpaper, 10 pt, conference]{ieeeconf}

\IEEEoverridecommandlockouts

\usepackage{amsmath,amssymb,amsfonts}
\usepackage{amssymb}
\usepackage{mathtools}
\usepackage{floatrow}
\usepackage{tikz}
\usepackage{tikz-3dplot} 
\usetikzlibrary{calc,arrows.meta} 
\usepackage{url}
\usepackage{pgfplots}
\usepackage{hyperref}
\usepackage{verbatim}   
\usepackage{ntheorem}
\usepackage{algorithm,algorithmic}
\usepackage{caption}

\usepackage{siunitx}
\usepackage{tabularx}
\usepackage{pgfplots}
\usepackage{pgfplotstable}
\usepackage{multirow}
\usepackage{cite}

\newtheorem{theorem}{Theorem}

\newtheorem{definition}{Definition}
\newtheorem{remark}{Remark}
\newtheorem{proposition}{Proposition}
\newtheorem{lemma}{Lemma}

\newcommand{\g}{\mathfrak{g}}
\newcommand{\h}{\mathfrak{d}}

\newcommand{\D}{\mathcal{D}}

\newtheorem{example}{Example}

\newcommand{\R}{\mathbb{R}}

\newcommand{\grad}{\normalfont{\text{grad}}}

\newcommand{\SO}{\normalfont{\text{SO}}}
\newcommand{\SE}{\normalfont{\text{SE}}}

\newcommand{\ad}{\normalfont{\text{ad}}}

\ifpdf
\hypersetup{
  pdftitle={Virtual constraints on Lie groups},
  pdfauthor={Efstratios, Alex, Leonardo, Anthony}
}
\fi

\def\BibTeX{{\rm B\kern-.05em{\sc i\kern-.025em b}\kern-.08em
    T\kern-.1667em\lower.7ex\hbox{E}\kern-.125emX}}
\begin{document}

\title{Virtual Holonomic and Nonholonomic Constraints on Lie groups}

\author{Alexandre Anahory Simoes$^{1}$, Anthony Bloch, Leonardo J. Colombo, Efstratios Stratoglou.
\thanks{A. Anahory Simoes is with the School of Science and Technology, IE University, Spain. (e-mail: alexandre.anahory@ie.edu).}
\thanks{A. Bloch is with Department of Mathematics, University of Michigan, Ann Arbor, MI 48109, USA. (e-mail: abloch@umich.edu)}
\thanks{L. Colombo is with Centre for Automation and Robotics (CSIC-UPM), Ctra. M300 Campo Real, Km 0,200, Arganda del Rey - 28500 Madrid, Spain. (e-mail: leonardo.colombo@csic.es).}
\thanks{E. Stratoglou is with American College of Thessaloniki (ACT), 17 Sevenidi St.
55535, Pylaia, Thessaloniki, Greece. (e-mail: stratogl@act.edu).}\thanks{The authors acknowledge financial support from Grant PID2022-137909NB-C21 funded by MCIN/AEI/ 10.13039/501100011033. A.B. was partially supported by NSF grant  DMS-2103026, and AFOSR grants FA9550-22-1-0215 and FA 9550-23-1-0400. The research leading to these results was supported in part by iRoboCity2030-CM, Robótica Inteligente para Ciudades Sostenibles (TEC-2024/TEC-62), funded by the Programas de Actividades I+D en Tecnologías en la Comunidad de Madrid.}}

\maketitle


\begin{abstract}
This paper develops a geometric framework for virtual constraints on Lie groups, with emphasis on mechanical systems modeled as affine connection systems.
Virtual holonomic and virtual nonholonomic constraints, including linear and affine nonholonomic constraints, are formulated directly at the level of the Lie algebra and characterized as feedback--invariant manifolds.
For each class of constraint, we establish existence and uniqueness conditions for enforcing feedback laws and show that the resulting closed--loop trajectories evolve as the dynamics of mechanical systems endowed with induced constrained connections, generalizing classical holonomic and nonholonomic reductions.
Beyond stabilization, the framework enables the systematic generation of low--dimensional motion primitives on Lie groups by enforcing invariant, possibly affine, manifolds and shaping nontrivial dynamical regimes.
The approach is illustrated through representative examples, including quadrotor UAVs and a rigid body with an internal rotor, where classical control laws are recovered as special cases and affine constraint--induced motion primitives are obtained.
\end{abstract}



\section{Introduction}

Mechanical systems whose configuration evolves on Lie groups arise naturally in robotics and aerospace applications, including spacecraft, underwater vehicles, and quadrotor UAVs \cite{bullo,lee2010geometric}.  
In such systems the configuration space is globally described by a matrix Lie group (typically $SO(3)$ or $SE(3)$), enabling intrinsic, coordinate-free control laws that avoid singularities and remain valid over the entire configuration space.  
When symmetries are present, geometric reduction on the associated Lie algebra further simplifies the dynamics and yields compact and structure-preserving control formulations.

Virtual constraints—relations enforced through feedback rather than physical mechanisms—provide a powerful tool for shaping the motion of underactuated mechanical systems.  
\emph{Virtual holonomic constraints} (VHCs) impose configuration-level relations and have been extensively used to encode gaits and low-dimensional motion primitives in legged robots \cite{chevallereau2009asymptotically,la2013stable,razavi2016symmetric,freidovich2008periodic,westervelt2018feedback}.  
\emph{Virtual nonholonomic constraints} (VNCs), instead, impose velocity-level relations; originally developed for walking robots \cite{griffin2015nonholonomic,griffin2017nonholonomic,horn2018hybrid,horn2020nonholonomic,hamed2019nonholonomic}, they extend the scope of virtual constraints to momentum-based and drift-dominated behaviors in underactuated systems.

While VHCs and VNCs are well understood in the tangent-bundle formulation of controlled Lagrangian systems \cite{isidori,bloch2003nonholonomic}, a unified geometric description of both classes of constraints \emph{on Lie groups} is still incomplete.  
Recent works provide first steps \cite{simoes2023virtual,stratoglou2023virtual,VC,moran2021energy}, but a comprehensive treatment of invariant manifolds (linear and affine), induced constrained connections, and the relation between holonomic/nonholonomic reduction and virtual constraint realization has been missing. The work \cite{stratoglou2026symplectic} elaborates virtual nonlinear nonholonomic constraints from a symplectic perspective where the trajectories that satisfy the constraints are characterized in terms of the almost-tangent and a symplectic structure on the tangent bundle.

This paper develops a geometric framework for both virtual holonomic and virtual nonholonomic constraints on Lie groups.  
We characterize linear and affine virtual constraints as feedback–invariant manifolds of the Lie algebra associated with an affine connection mechanical system.  
For both VHCs and VNCs we provide sufficient conditions ensuring the existence and uniqueness of enforcing feedback laws, and we show that the resulting closed-loop trajectories correspond to geodesics of an induced constrained connection.  
Because the constrained closed-loop dynamics evolve on low-dimensional invariant manifolds or distributions, the resulting behaviours constitute geometric motion primitives on Lie groups, providing a systematic tool for encoding and designing families of maneuvers. In both the holonomic and nonholonomic settings, we clarify under which conditions virtual constraints reproduce the reduced constrained dynamics of classical mechanical systems. In the holonomic case, virtual holonomic constraints recover the reduced Euler–Lagrange equations on the constraint submanifold, while in the nonholonomic case we further identify precise conditions under which virtual nonholonomic constraints recover the reduced dynamics of classical nonholonomic systems \cite{bloch2003nonholonomic}.

Beyond the theoretical contributions, this framework is directly relevant for robotic and mechanical systems evolving on $SE(3)$. Virtual holonomic constraints enable enforcing configuration-based behaviors such as attitude alignment or motion along Lie subgroups, while virtual nonholonomic constraints allow encoding velocity-dependent navigation primitives, coordinated planar motion, or thrust-attitude coupling in a way that respects the geometric structure and underactuation.

The paper is organized as follows. Section \ref{Sec: background} reviews the differential geometric preliminaries on manifolds and Lie groups used throughout the paper. Section \ref{sec3} develops the Riemannian and Levi-Civita connection framework for mechanical systems on Lie groups, including their reduction to the Lie algebra.
Section \ref{sec4} introduces virtual holonomic constraints on Lie groups, while Section \ref{sec5} develops the corresponding theory of virtual nonholonomic constraints, including both linear and affine cases.  
In both sections, we establish existence and uniqueness results for enforcing feedback laws, derive the induced constrained connections governing the closed--loop dynamics on the associated invariant manifolds, and illustrate the theory through representative examples.

\section{Background on Differential Geometry}\label{Sec: background}

\subsection{Background on vector bundles and vector fields}
A vector bundle of rank $k>0$ on the manifold $Q$ is a smooth assignment to each point $q\in Q$ of a vector space with dimension $k$. Relevant particular cases for us are tangent bundles, where to each point $q$ we assign the tangent space at $q$; and distributions, in which the vector space is a subspace of the tangent space. 

If $P$ is a vector bundle on $Q$, then $P$ might be written as a collection of vector spaces $P=\cup_{q\in Q} P_{q}$, where $P_{q}$ is the $k$-dimensional space assigned to the point $q$. In addition, there is a map called the bundle projection and denoted by $\pi:P\to Q$, defined by $\pi(p)=q$ if $p\in P_{q}$. Smooth sections of a vector bundle $P$ on $Q$ are smooth maps $X:Q\to P$ having the property that $X(q)\in P_{q}$, i.e., the vector $X(q)$ must belong to the vector space assigned to $q$ for all $q\in Q$. We denote the collection of smooth section of a vector bundle $P$ by $\Omega(P)$.

Vector fields are smooth sections of the tangent bundle. Vector fields are smooth maps of the form $X: Q \to TQ$ such that $\pi \circ X = \text{id}_Q$, the identity function on $Q$. For a vector field $X\in \mathfrak{X}(Q)$ we define its vertical lift to $TQ$ to be $\displaystyle{X_{v_{q}}^{V}=\left. \frac{d}{dt}\right|_{t=0} (v_{q} + t X(q))}$, where $v_q\in T_qQ$. This defines a vector field $X^V$ on $TQ$, called the vertical lift of $X$. 

The flow of a vector field $X\in \mathfrak{X}(Q)$ is the map $\phi^{X}_{t}:Q \to Q$ such that the curve $q(t)=\phi_{t}(q_0)$ is the solution of the integral curve of the differential equation $\dot{q}=X(q(t))$ with initial condition $q(0)=q_{0}$ for each fixed $q_0\in Q$.

\subsection{Background on Riemannian Manifolds}
Let $(Q, \left< \cdot, \cdot\right>)$ be an $n$-dimensional \textit{Riemannian
manifold}, where $Q$ is an $n$-dimensional smooth manifold and $\left< \cdot, \cdot \right>$ is a positive-definite symmetric covariant 2-tensor field called the \textit{Riemannian metric}. That is, to each point $q\in Q$ we assign a positive-definite inner product $\left<\cdot, \cdot\right>_q:T_qQ\times T_qQ\to\mathbb{R}$, where $T_qQ$ is the \textit{tangent space} of $Q$ at $q$ and $\left<\cdot, \cdot\right>_q$ varies smoothly with respect to $q$. The length of a tangent vector is determined by its norm, defined by
$\|v_q\|=\left<v_q,v_q\right>^{1/2}$ with $v_q\in T_qQ$.  The cotangent bundle of $Q$ is denoted by $T^*Q$ and for $q\in Q$ the cotangent space $T^*_qQ$, which is the dual of the tangent space $T_qQ$. For any $q \in Q$, the Riemannian metric induces an invertible map $\flat: T_q Q \to T_q^\ast Q$, called the \textit{flat map}, defined by $\flat(X)(Y) = \left<X, Y\right>$ for all $X, Y \in T_q Q$. The inverse map $\sharp: T_q^\ast Q \to T_q Q$, called the \textit{sharp map}, is similarly defined implicitly by the relation $\left<\sharp(\alpha), Y\right> = \alpha(Y)$ for all $\alpha \in T_q^\ast Q$. Let $C^{\infty}(Q)$ and $\mathfrak{X}(Q)$ denote the spaces of smooth functions and smooth vector fields on $Q$, respectively. The sharp map provides a map from $C^{\infty}(Q) \to \mathfrak{X}(Q)$ via $\grad f(q) = \sharp(df_q)$ for all $q \in Q$, where $\grad f$ is called the \textit{gradient vector field} of $f \in C^{\infty}(Q)$. This identification allows expressing force fields derived from a potential 
$V \in C^{\infty}(Q)$ as vector fields of the form $-\grad V = -\,\sharp(dV)$, 
which naturally appear in the Euler–Lagrange and Euler–Poincaré equations as 
configuration–dependent drifts.

An \textit{affine connection} on $Q$ is a map $\nabla: \mathfrak{X}(Q) \times \mathfrak{X}(Q) \to \mathfrak{X}(Q)$ which is $C^{\infty}(Q)$-linear in the first argument, $\R$-linear in the second argument, and satisfies the product rule $\nabla_X (fY) = X(f) Y + f \nabla_X Y$ for all $f \in C^{\infty}(Q), \ X, Y \in \mathfrak{X}(Q)$. The connection plays a role similar to that of the directional derivative in classical real analysis. The operator
$\nabla_{X}$ which assigns to every smooth vector field $Y$ the vector field $\nabla_{X}Y$ is called the \textit{covariant derivative} (of $Y$) \textit{with respect to $X$}.

We will use the Levi-Civita connection defined by the following formula known as the \textit{Koszul formula}:
\begin{align}\label{eq:Koszul}
2 \, \langle \nabla_X Y , Z \rangle 
&= X \big( \langle Y , Z \rangle \big)
 + Y \big( \langle X , Z \rangle \big)
 - Z \big( \langle X , Y \rangle \big) \nonumber \\
& - \langle X , [Y,Z] \rangle
 - \langle Y , [X,Z] \rangle
 + \langle Z , [X,Y] \rangle .
\end{align}

Let $\gamma: I \to Q$ be a smooth curve parameterized by $t \in I \subset \R$, and denote the set of smooth vector fields along $\gamma$ by $\Gamma(\gamma)$. Then for any affine connection $\nabla$ on $Q$, there exists a unique operator $\nabla_{\dot{\gamma}}: \Gamma(\gamma) \to \Gamma(\gamma)$ (called the \textit{covariant derivative along $\gamma$}) which agrees with the covariant derivative $\nabla_{\dot{\gamma}}\tilde{W}$ for any extension $\tilde{W}$ of $W$ to $Q$, with $W,\tilde{W}\in\Gamma(\gamma)$. 
 The covariant derivative allows  one to define a particularly important family of smooth curves on $Q$ called \textit{geodesics}, which are defined as the smooth curves $\gamma$ satisfying $\nabla_{\dot{\gamma}} \dot{\gamma} = 0$. 

It is well-known that the Riemannian metric induces a unique torsion-free and metric compatible connection called the \textit{Riemannian connection}, or the \textit{Levi-Civita connection}. In mechanical systems endowed with a Riemannian metric, the Levi–Civita 
connection determines the notion of kinetic acceleration and the evolution 
of free motion, while potential forces arise from gradients of smooth 
functions $V:Q\to\mathbb{R}$ under the same metric pairing. In the remainder of this paper, we will assume that $\nabla$ is the Riemannian connection. For additional information on connections, we refer the reader to \cite{Boothby,Milnor}. When the covariant derivative corresponds to the Levi-Civita connection, geodesics can also be characterized as the critical points of the length functional $\displaystyle{L(\gamma) = \int_0^1 \|\dot{\gamma}\|dt}$ among all unit-speed \textit{piecewise regular} curves $\gamma: [a, b] \to Q$ (that is, where there exists a subdivision of $[a, b]$ such that $\gamma$ is smooth and satisfies $\dot{\gamma} \ne 0$ on each subdivision). 

The metric $\langle \cdot , \cdot \rangle$ defines a Riemannian distance function
\[
d(x,y) = \inf_{\gamma} \int_0^1 \|\dot{\gamma}(t)\|\, dt,
\]
where the infimum is taken over all regular curves $\gamma : [0,1] \to Q$ joining $x$ and $y$. 
This expression is invariant under reparameterization of $\gamma$. 
If $\gamma_{x,y}$ denotes the minimizing geodesic between $x$ and $y$, then 
\[
d(x,y) = \int_0^1 \|\dot{\gamma}_{x,y}(t)\|\, dt .
\] If we assume that $Q$ is \textit{complete} (that is, $(Q, d)$ is a complete metric space), then by the Hopf-Rinow theorem, any two points $x$ and $y$ in $Q$ can be connected by a (not necessarily unique) minimal-length geodesic $\gamma_{x,y}$. In this case, the Riemannian distance between $x$ and $y$ can be defined by $\displaystyle{d(x,y)=\int_{0}^{1}\Big{\|}\frac{d \gamma_{x,y}}{d s}(s)\Big{\|}\, ds}$. Moreover, if $y$ is contained in a geodesically convex neighborhood of $x$ (that is, the point $y$ must
belong to a convex open ball around $x$), we can write the Riemannian distance by means of the Riemannian exponential as $d(x,y)=\|\mbox{exp}_x^{-1}y\|$ (see \cite{bullo2019geometric} for instance). 

\section{Euler--Poincar\'e equations}\label{sec3}

Let $G$ be a Lie group with Lie algebra $\g := T_eG$. The left-translation map
$L_g : G \to G$, $L_g h = gh$, is a diffeomorphism for all $g\in G$. 
Given an inner product $\langle\cdot,\cdot\rangle_{\g}$ on the Lie algebra $\g$, 
the left-translation map $L_g : G \to G$, $L_g h = gh$, induces a left-trivialization 
of tangent vectors via its tangent map $(L_g)_\ast : T_h G \to T_{gh} G$.  
Using this, we define a left-invariant Riemannian metric on $G$ by
\[
\langle X_g, Y_g\rangle 
:= \big\langle (L_{g^{-1}})_\ast X_g,\; (L_{g^{-1}})_\ast Y_g\big\rangle_{\g}
\] for $g\in G,\; X_g,Y_g\in T_gG$. This establishes a one-to-one correspondence between inner products on $\g$ and 
left-invariant Riemannian metrics on $G$, and ensures that each left translation 
$L_g$ is an isometry \cite{goodman2022reduction}.

A vector field $X$ on $G$ is \emph{left-invariant} if $(L_g)_\ast X = X$ for all $g\in G$.
We denote by $\mathfrak{X}_L(G)$ the space of left-invariant vector fields. The map $(\cdot)_{L}:\g\to \mathfrak{X}_L(G)$,
\[
\xi \mapsto \xi_{L} (g) = (L_g)_\ast \xi
\]
is a vector space isomorphism and, with the usual Lie bracket of vector fields, a Lie algebra
isomorphism.

Let $\nabla$ be the Levi-Civita connection on $(G,\langle\cdot,\cdot\rangle)$.
The isomorphism $\varphi$ induces a bilinear map $\nabla^\g:\g\times\g\to\g$,
\[
\nabla^\g_\xi\eta := \big(\nabla_{\xi_{L}}\eta_{L}\big)(e),
\]
called the \emph{Riemannian $\g$-connection}. It is $\mathbb{R}$-bilinear and satisfies, 
for all $\xi,\eta,\sigma\in\g$,
\begin{align}
    \nabla^\g_\xi\eta - \nabla^\g_\eta\xi &= [\xi,\eta], \label{eq:g-conn-torsion}\\
    \langle\nabla^\g_\sigma\xi,\eta\rangle_{\g}
    +\langle\xi,\nabla^\g_\sigma\eta\rangle_{\g} &= 0. \label{eq:g-conn-metric}
\end{align}
We have the explicit formula (cf.\ Theorem~5.40 in
\cite{bullo2019geometric})
\begin{equation}\label{eq:g-connection-formula}
    \nabla^\g_\xi\eta
    =\frac12\Big([\xi,\eta]
    -\sharp\left[\ad_\xi \flat(\eta)\right]
    -\sharp \left[\ad_\eta \flat(\xi)\right]\Big).
\end{equation}

Let $g:[a,b]\to G$ be a smooth curve, and let $X$ be a smooth vector field along $g$.
Define $\xi(t) := g(t)^{-1}\dot g(t)\in\g$, 
$\eta(t) := g(t)^{-1}X(t)\in\g$. Then (see  \cite{goodman2022reduction})
\begin{equation}\label{eq:cov-to-g-conn}
    \nabla_{\dot g}X(t)
    = (L_{g(t)})_\ast\big(\dot\eta(t)+\nabla^\g_{\xi(t)}\eta(t)\big).
\end{equation}

\begin{theorem}[\cite{goodman2022reduction}]
Let $g:[a,b]\to G$ be a solution of the Euler--Lagrange equations for the 
left-invariant mechanical Lagrangian $L(g,\dot g)=\tfrac12\|\dot g\|^2 - V(g)$, and define the left-trivialized velocity $\xi := g^{-1}\dot g \in \g$. Then $\xi:[a,b]\to\g$ satisfies the Euler--Poincar\'e equation 
\begin{equation}\label{eq:EP-geo}
    \dot\xi + \nabla^\g_\xi\xi 
    + (L_{g^{-1}})_{\ast}\big(\grad V(g)\big) = 0,
\end{equation}
or, using \eqref{eq:g-connection-formula},
\begin{equation}\label{eq:EP-geo-adstar}
    \dot\xi 
    - \sharp\big(\ad^\ast_\xi\flat(\xi)\big) 
    + (L_{g^{-1}})_{\ast}\big(\grad V(g)\big)
    = 0.
\end{equation}
\end{theorem}

The second and third-terms in equation \eqref{eq:EP-geo} are the 
drift terms of the Euler--Poincar\'e equations. When the Riemannian connection and the potential function are clear from the context, we will use the following notation for the drift terms
\[
B(\xi,g):=\nabla^{\g}_{\xi}\xi + (L_{g^{-1}})_\ast(\grad V(g)),
\]

\begin{remark}
Given a left-invariant Lagrangian $L:T G\to\mathbb{R}$ satisfying 
$L(hg,h\dot g)=L(g,\dot g)$, the reduced Lagrangian $l:\g\to\mathbb{R}$ is defined by
$l(\xi)=L(e,\xi)=L(g,\dot g)$ for $\xi=g^{-1}\dot g$.
The Euler--Poincar\'e equations for $\ell$ are then
 \[
\frac{d}{dt}\Big(\frac{\partial l}{\partial\xi}\Big)
= \ad^\ast_\xi\frac{\partial l}{\partial\xi}
    + \flat\left[(L_{g^{-1}})_\ast\!\big(\grad V(g)\big)\right] .
\]

For the mechanical Lagrangian $L(g,\dot g)=\tfrac12\|\dot g\|^2 - V(g)$,
we have $l(\xi)=\tfrac12\|\xi\|^2 - V(g)$, 
$\partial l/\partial\xi=\flat[\xi]$, and therefore
\[
\flat(\dot\xi) 
= \ad^\ast_\xi\flat(\xi) 
    + \flat\left[(L_{g^{-1}})_\ast\!\big(\grad V(g)\big)\right],
\]
which is equivalent to \eqref{eq:EP-geo-adstar} after applying the sharp map.
\end{remark}

Now, let us assume that $(G,\langle\cdot,\cdot\rangle)$ is a Lie group with a left-invariant Riemannian metric,
and let $K\subset G$ be a Lie subgroup with Lie algebras $\mathfrak{g}$ and $\mathfrak{k}$, respectively. 
The metric restricts to a left-invariant metric on $K$, with Levi-Civita connection
$\nabla^K$ and associated Riemannian $\mathfrak{k}$-connection 
$\nabla^{\mathfrak{k}}:\mathfrak{k}\times\mathfrak{k}\to\mathfrak{k}$ defined
as above. Denote by $\mathfrak{k}^\perp\subset\mathfrak{g}$ 
the orthogonal complement of $\mathfrak{k}$.

A curve $g:I\to K$ is a geodesic on $K$ if and only if
\[
\nabla^G_{\dot g}\dot g \;\in\; (T K)^\perp,
\qquad
\dot g\in T_gK,
\]
where $\nabla^G$ is the Levi-Civita connection on $G$ and $(T K)^\perp$ is the orthogonal
distribution to $T K$ in $T G$.

\begin{theorem}\label{thm:holonomic-EP-left}
Let $g:I\to K\subset G$ be a smooth curve and define the left-trivialized velocity
$\xi := g^{-1}\dot g\in\mathfrak{k}$. Consider the mechanical Lagrangian
$L(g,\dot g)=\tfrac12\|\dot g\|^2 - V(g)$,
and let $V_K$ be its restriction to $K$. Then the following statements are equivalent:
\begin{enumerate}
    \item $g$ is a solution of the constrained Euler--Lagrange equations on $K$, i.e.
    \begin{equation}\label{holonomic:EL}
        \nabla^K_{\dot g}\dot g + \grad_K V_K(g) = 0.
    \end{equation}

    \item The left-trivialized velocity $\xi$ satisfies the forced Euler--Poincar\'e equation
    on the Lie algebra $\mathfrak{k}$ of $K$:
    \begin{equation}\label{holonomic:potential:dynamics}
        \dot\xi + \nabla^{\mathfrak{k}}_{\xi}\xi 
        + (L_{g^{-1}})_\ast\big(\grad_K V_K(g)\big)
        = 0 .
    \end{equation}

    \item Viewed in the full Lie algebra $\mathfrak{g}$, $\xi$ satisfies the constrained
    Euler--Poincar\'e condition:
    \begin{equation}\label{holonomic:constrained:g}
        \dot\xi + \nabla^{\mathfrak{g}}_{\xi}\xi 
        + (L_{g^{-1}})_\ast\big(\grad V(g)\big)
        \;\in\; \mathfrak{k}^\perp .
    \end{equation}
\end{enumerate}
\end{theorem}

\section{Virtual holonomic constraints on Lie groups}\label{sec4}

Consider a mechanical controlled system determined by a Lagrangian of the form kinetic minnus potential energy $L(g,\dot g)=\tfrac12\|\dot g\|^2 - V(g)$, where the control force $F:TQ\times U \rightarrow T^{*}Q$ is of the form
\begin{equation}\label{force:control}
    F(q,\dot{q},u) = \sum_{a=1}^{m} u^{a}f^{a}(q)
\end{equation}
where $f^{a}$ is a one-form on $Q$ with $m<n$, $U\subset\mathbb{R}^{m}$ the set of controls and $u^a\in\mathbb{R}$ with $1\leq a\leq m$ the control inputs. Then, the Riemannian form of the equations of motion reads
\begin{equation}\label{mechanical:control:system}
    \nabla_{\dot{q}(t)} \dot{q}(t) + \grad \ V (q(t))= u^{a}(t)Y_{a}(q(t)),
\end{equation}
with $Y_{a}=\sharp(f^{a})$ the corresponding force vector fields. The distribution $\mathcal{F}\subseteq TQ$ generated by the vector fields  $Y_{a}$ is called the \textit{input distribution} associated with the mechanical control system \eqref{mechanical:control:system}.

A \textit{feedback control law} for equation \eqref{mechanical:control:system} is a set of $m$ functions $\hat{u}=(\hat{u}^{1}, \ldots, \hat{u}^{m}):\mathcal{U}\subseteq TQ \rightarrow \mathbb{R}^{m}$. Substituting the control law as the control input in \eqref{mechanical:control:system}, we obtain an equation without controls that is called \textit{closed-loop system}.


Similarly to the uncontrolled equation, we may associate to the controlled system \eqref{mechanical:control:system} the vector field $\Gamma_{u}\in \mathfrak{X}(TQ)$ of the form
$$\Gamma_{u}(v_{q}) = \Gamma^{\nabla}(v_{q}) - (\grad \ V)^{V}_{v_{q}} - u^{a} (Y_{a})^{V}_{v_{q}},$$ where $\Gamma^{\nabla} \in \mathfrak{X}(TQ)$ is the \textit{geodesic vector field} of the connection $\nabla$.

Let $\hat{u}=(\hat{u}^{1}, \ldots, \hat{u}^{m}):\mathcal{U}\subseteq TQ \rightarrow \mathbb{R}^{m}$ be a feedback control law defined on an open subset $\mathcal{U}\subseteq TQ$. By choosing this control input in system \eqref{mechanical:control:system}, we obtain a closed-loop system whose solution $q(t)$ with initial point $q\in Q$ and initial velocity $v_{q}\in T_{q}Q$ is connected to the flow of the vector field $\Gamma_{\hat{u}}$ through the relationship $$\phi_{t}^{\hat{u}}(v_{q})=(q(t), \dot{q}(t)),$$ where $\phi_{t}^{\hat{u}}:\mathcal{U} \to \mathcal{U}$ is the flow of $\Gamma_{\hat{u}}$ and it is well-defined at least for small values of $t$.

In the following, if $\hat{u}$ denotes a control law for system \eqref{mechanical:control:system} then $\phi_{t}^{\hat{u}}$ denotes the flow of the corresponding closed-loop system.

\begin{definition}\label{control:invariant:definition}
    An embedded submanifold $\mathcal{N}\subseteq Q$ is said to be controlled invariant for the control system \eqref{mechanical:control:system} if there exist $m$ control functions $\hat{u}=(\hat{u}^{1}, \ldots, \hat{u}^{m}):T\mathcal{N}\rightarrow \mathbb{R}^{m}$ such that the flow of the closed-loop system $\phi_{t}^{\hat{u}}$ is defined on an open set $\mathcal{B}$ of $T\mathcal{N}$ containing the zero vector and satisfies $\phi_{t}^{\hat{u}}(\mathcal{B})\subseteq T\mathcal{N}$ for all $t\geqslant 0$.
\end{definition}

The concept of virtual holonomic constraints is based on that of controlled invariant submanifold.

\begin{definition}
A \textit{virtual holonomic constraint} associated with the mechanical control system \eqref{mechanical:control:system} is a controlled invariant submanifold $\mathcal{N}\subseteq Q$ for that control system. 
\end{definition}

Provided that $\mathcal{F}$ and $\hbox{ker}(T\Phi)$ are transversal, i.e., their intersection is zero and they span the entire tangent bundle, there is a unique control law for the system \eqref{mechanical:control:system} making $\mathcal{N}$ a virtual holonomic constraint, see \cite{consolini2018coordinate}, \cite{consolini2015induced}.

\begin{remark}\label{remarkzd}
    The transversality assumption is equivalent to the assumption of (vector) relative degree $\{1,\ldots,1\}$ appearing in the literature of zero dynamics (see \cite{isidori1985nonlinear}) concerning control systems evolving in Euclidean spaces. It is simple to show that if $Y_{a}\in \mathcal{F}$ are the vectors spanning the input distribution, then the relative degree of $\Phi$ is $\{1,\ldots,1\}$ if $\langle d\Phi(v_{q}), ((Y)_{a})_{v_{q}}^{V} \rangle \neq 0$ for all $a$. This is equivalent to our transversality assumption. 
\end{remark}

\begin{remark}
    If $\mathcal{F}$ and $\hbox{ker}(T\Phi)$ span the entire tangent bundle but their intersection is not just the zero vector at each point, then a control law still exists but it may not be unique anymore (see \cite{simoes2023virtual} for examples).
\end{remark}

\subsection{Virtual holonomic constraints on Lie Groups}\label{vclg}

Given a control force $F:\mathfrak{g}\times U \rightarrow \mathfrak{g}^{*}$ of the form
\begin{equation}\label{control:force:general}
    F(\xi,u) = \sum_{a=1}^{m} u^{a}F^{a}(\xi)
\end{equation}
where $F^{a}\in\mathfrak{g}^{*}$ with $m<n$, $U\subset\mathbb{R}^{m}$ the set of controls and $u^a\in\mathbb{R}$ with $1\leq a\leq m$ the control inputs, consider a controlled mechanical system on the Lie group $G$ of the form
\begin{equation}\label{control:system}
    \dot{g} = T_{e}L_{g}(\xi), \quad \dot{\xi} - \sharp \left[ \text{ad}^{*}_{\xi(t)}\flat(\xi(t))\right] + \grad_{b} V(g) = u^{a}f_{a},
\end{equation}
where $f_{a}=\sharp(F^a|_\g)\in \g$ are $m<n$ vectors spanning the control input subspace
$\mathfrak{f} = \text{span} \{f_{1}, \dots, f_{m}\}$ and $\grad_{b} V(g) = T_g L_{g^{-1}}\big(\grad V(g)\big)$ is the gradient of the potential in the body frame.

\begin{definition}
  The subspace $\mathfrak{f}$ of the Lie algebra $\g$ given above is called the  \textit{control input subspace} associated with the mechanical control system \eqref{control:system}. 
\end{definition}


\begin{definition}A Lie subalgebra $\mathfrak{k}$ of $\g$ is a subspace of $\g$ that is closed under the Lie bracket, i.e., $[\xi, \eta]\in \mathfrak{k}$ for all $\xi, \eta\in\mathfrak{k}$.\end{definition}

\begin{definition}
   A virtual holonomic constraint associated with the mechanical system of type \eqref{control:system} is a controlled invariant Lie subalgebra $\mathfrak{k}$ of $\g$, that is, there exists a control law  making the Lie subalgebra $\mathfrak{k}$ invariant under the flow of the closed-loop system, i.e. $\xi(0)\in\mathfrak{k}$ and $\xi(t)\in\mathfrak{k}, \; \forall t\geq 0$.
    
\end{definition}

Note that if $\mathfrak{k}$ is a Lie subalgebra, then by Lie's third theorem (see Remark 5.29 in \cite{bullo}), there exists a unique connected Lie subgroup $K\subseteq G$ with Lie algebra $\mathfrak{k}$.

In addition, if the control system \eqref{control:system} satisfies $\xi(t)\in\mathfrak{k}$, then the curve $g$ is  by definition tangent to the left-invariant integrable distribution defined by $T_e L_g(\mathfrak{k})=T_e L_g(T_{e}K)$ which coincides with $T_g K$ for each $g\in K$. In particular, if the curve $g(t)$ starts at a point $g_0\in K$, it will be constrained to the subgroup $K$ for all time $t$.

\begin{theorem}\label{maintheorem}
    Suppose $\mathfrak{k}$ and $\mathfrak{f}$ are subspaces satisfying $\g = \mathfrak{f}\oplus \mathfrak{k} $. Then there exists a unique control law $u^{*}$ making $\mathfrak{k}$ a virtual contraint for the controlled mechanical system \eqref{control:system}.
\end{theorem}

\begin{proof}
    Let $\dim \mathfrak{k} = k$ and $\dim \mathfrak{f}=m=n-k$. Consider the covectors $\mu^{1}\dots, \mu^{m} \in \g^{*}$ spanning the annihiliator subspace of $\mathfrak{k}$. If $\xi(t)$ is a curve on $\g$, it satisfies  $\xi(t)\in \mathfrak{k}$ for all time if and only if $\mu^{a}(\xi(t))= 0$ for all $a=1,\dots, m$. Differentiating this equation and assuming that $\xi(t)$ is a solution of the closed loop system \eqref{control:system} for an appropriate choice of control law $u$, we have that
    $$\mu^{a}\left( \sharp \left[ \text{ad}^{*}_{\xi(t)}\flat(\xi(t))\right] - \grad_{b} V(g(t)) \right) + u^{b}\mu^{a}(f_{b}) = 0.$$
    Since $\sharp \left[ \text{ad}^{*}_{\xi(t)}\flat(\xi(t))\right]- \grad_{b} V(g(t)) \in \g$ and $\g = \mathfrak{f}\oplus \mathfrak{k} $ there is a unique way to decompose this vector as the sum
    $$\sharp \left[ \text{ad}^{*}_{\xi(t)}\flat(\xi(t))\right]- \grad_{b} V(g(t)) = \eta(t) + \tau^{b}(t)f_{b},$$
    with $\eta \in \mathfrak{k}$. In addition, note that the coefficients $\tau^{b}$ may be regarded as a function on $\g$. In fact, its definition is associated with the projection to $\mathfrak{f}$ together with the choice of $\{f_{b}\}$ as a basis for $\mathfrak{f}$. Therefore, $\mu^{a}(\xi(t))=0$ if and only if
    $$(\tau^{b}-u^{b}) \mu^{a}(f_{b})=0.$$
    Since $\mu^{a}(f_{b})$ is an invertible matrix, we conclude that $\tau^{b}=u^{b}$, proving existence and uniqueness of a control law making $\mathfrak{k}$ a virtual constraint.
\end{proof}

\begin{remark}
In the context of Remark \ref{remarkzd}, Theorem \ref{maintheorem} is a geometric generalization of Propositon 6.1.2. in \cite{isidori1985nonlinear} applied to simple mechanical control systems on a Lie group endowed with a Riemannian structure.\end{remark}

From now on suppose that the subspace $\mathfrak{k}$ describing the virtual constraints 
and the input subspace $\mathfrak{f}$ are transversal. Therefore, the projections 
$\text{pr}_{\mathfrak{k}}:\g\rightarrow \mathfrak{k}$ and $\text{pr}_{\mathfrak{f}}:\g\rightarrow \mathfrak{f}$ 
associated with the direct sum are well-defined. Using the Riemannian $\g$-connection we define the bilinear map $\nabla^{\mathfrak{c}}$ 
\begin{equation}\label{virtual:nh:connection}
    \nabla^{\mathfrak{c}}_{\xi} \eta = \nabla^{\g}_{\xi} \eta + (\nabla_{\xi}^{\g}\text{pr}_{\mathfrak{f}})(\eta),
\end{equation} 
where $\xi,\eta\in\g$. We refer to that bilinear map as the  \textit{induced constrained connection} associated to the subspace $\mathfrak{k}$ and the input subspace $\mathfrak{f}$. The induced constrained connection is a bilinear map on $\g$ with the special property that restricts to $\mathfrak{k}$, i.e., $\nabla^{\mathfrak{c}}_\xi \eta \in \mathfrak{k}$ for $\xi,\eta\in \mathfrak{k}$, as the next lemma shows.

\begin{lemma}\label{constrained:property}
    If $\xi,\eta \in \mathfrak{k}$ then
    $\nabla^{\mathfrak{c}}_{\xi} \eta = \text{pr}_{\mathfrak{k}}(\nabla^{\g}_\xi \eta) \in \mathfrak{k}.$
\end{lemma}
\begin{proof}
If $\xi,\eta \in \mathfrak{k}$ we have that
    \begin{equation*}
        \begin{split}
            \nabla^{\mathfrak{c}}_{\xi} \eta = & \nabla^{\g}_{\xi} \eta + (\nabla^{\g}_{\xi}\text{pr}_{\mathfrak{f}})(\eta) \\
             = & \nabla^{\g}_{\xi} \eta + \nabla^{\g}_{\xi}(\text{pr}_{\mathfrak{f}}(\eta))- \text{pr}_{\mathfrak{f}}(\nabla^{\g}_{\xi} \eta),
        \end{split}
    \end{equation*}
    where we have used the definition of covariant derivative of a map in the last equality. Noting that $\text{pr}_{\mathfrak{f}}(\eta)=0$ since $\eta\in\mathfrak{k}$, we prove the equality.
\end{proof}

The control law from Theorem \ref{maintheorem} induces a closed-loop system that might be written in a coordinate free manner using the induced connection as the next result shows.

\begin{theorem}\label{orthogonal:input:thm}
    Let $\xi(t)\in\mathfrak{k}$ be a solution of the closed-loop controlled mechanical system \eqref{control:system} with the unique control law $u^*$. Then, $\xi(t)$ satisfies the constrained dynamics
    $$\dot{\xi} + \nabla^{\mathfrak{c}}_{\xi} \xi + \text{pr}_{\mathfrak{k}}(\grad_{b} V(g)) = 0.$$
    In addition, if the control input subspace $\mathfrak{f}$ is orthogonal to the controlled invariant subspace $\mathfrak{k}$ with respect to the inner product on $\g$, then the constrained dynamics coincides with the holonomic constrained dynamics \eqref{holonomic:potential:dynamics} on points of the subgroup $K$.
\end{theorem}

\begin{proof}
 Firstly, rewrite the controlled system \eqref{control:system} in the form
    $$\dot{\xi} + \nabla^{\mathfrak{g}}_{\xi} \xi + \grad_{b} V(g)= u^{a}f_{a}.$$
    Inspecting the proof of Theorem \ref{maintheorem}, one recognizes that the control law $u^a$ is chosen so that the term $u^af_a$ is nothing but $\text{pr}_{\mathfrak{f}}(\nabla^{\mathfrak{g}}_{\xi} \xi + \grad_{b} V(g)$ reflecting the decomposition of $\g$ as the direct sum of $\mathfrak{k}$ and $\mathfrak{f}$. Thus, the closed-loop system is equivalent to
    $$\dot{\xi} + \nabla^{\mathfrak{g}}_{\xi} \xi = \text{pr}_{\mathfrak{f}}(\nabla^{\mathfrak{g}}_{\xi} \xi + \grad_{b} V(g)),$$
    or, equivalently,
    $$\dot{\xi} + \text{pr}_{\mathfrak{k}}(\nabla^{\mathfrak{g}}_{\xi} \xi + \grad_{b} V(g)) = 0,$$
    which by virtue of the previous Lemma gives
    $$\dot{\xi} + \nabla^{\mathfrak{c}}_{\xi} \xi  + \grad_{b} V(g)= 0.$$

    Secondly, if $\mathfrak{f}=\mathfrak{k}^{\bot}$, where the orthogonal complement is taken with respect to the inner product $\langle\cdot,\cdot\rangle$, then the orthgonal projections $\text{pr}_{\mathfrak{k}}$ and $\text{pr}_{\mathfrak{k}}$ coincide, as well as the projections $\text{pr}_{\mathfrak{k}^{\bot}}$ and $\text{pr}_{\mathfrak{f}}$. Thus, the induced constrained connection $\nabla^\mathfrak{c}$ is given by
    $$\nabla^\mathfrak{c}_\xi \xi = \text{pr}_{\mathfrak{k}} (\nabla^{\mathfrak{g}}_{\xi} \xi) = \nabla^{\mathfrak{k}}_{\xi} \xi.$$
    
    Therefore, the closed-loop dynamics satisfies
    $$\dot{\xi} + \text{pr}_{\mathfrak{k}} (\nabla^{\mathfrak{g}}_{\xi} \xi  + \grad_{b} V(g)) = 0$$
    which is equivalent to
    $\dot{\xi} + \nabla^{\mathfrak{k}}_{\xi} \xi  + \text{pr}_{\mathfrak{k}} (\grad_{b} V(g))= 0$
    for $\xi(t)\in \mathfrak{k}$. Using that $\text{pr}_{\mathfrak{k}} \circ T_{g}L_{g^{-1}} = T_{g}L_{g^{-1}} \circ \text{pr}_{TK}$, we deduce that
   $\text{pr}_{\mathfrak{k}} (\grad_{b} V(g)) = T_{g}L_{g^{-1}} \circ \text{pr}_{TK}(\grad V (g))$
    which is equivalent to
    $$\text{pr}_{\mathfrak{k}} (\grad_{b} V(g)) = T_{g}L_{g^{-1}} \grad_{K}V_{K}(g).$$
    on points of $g\in K$, which gives precisely equations \eqref{holonomic:potential:dynamics}.
\end{proof}

Extending the previous results, the following theorem formulates the virtual constraint in the 
orthogonal setting $\mathfrak{g} = \mathfrak{k} \oplus \mathfrak{f}$, 
making explicit the role of the projection $\text{pr}_{\mathfrak{f}}$ and the induced connection 
$\nabla^{\mathfrak{c}}$ on the constraint subspace.

\begin{proposition}\label{orthogonal:input:thm2}
Consider the decomposition $\mathfrak{g} = \mathfrak{k} \oplus \mathfrak{f}$, 
where $\mathfrak{k}$ and $\mathfrak{f}$ are complementary subspaces. 
Assume that the control directions $\{f_a\}_{a=1}^m$ form an orthonormal basis of $\mathfrak{f}$ and let $\{F^a\}_{a=1}^m\subset \mathfrak{g}^*$ be the dual basis of $\{f_a\}$, i.e., 
$F^a(f_b)=\delta^a_b$ and $F^a|_{\mathfrak{k}}=0$.

Then, the unique control input $u^\star$ that renders $\mathfrak{k}$ invariant under the closed–loop dynamics for the controlled Euler–Poincaré system \eqref{control:system} is
\[
u^\star_a = F^a\big(\text{pr}_{\mathfrak{f}}(\nabla^{\mathfrak{g}}_\xi \xi + \grad_{b} V(g)) \big)= F^a\big(\text{pr}_{\mathfrak{f}}(B(\xi,g)) \big), 
\]
where $\nabla^{\mathfrak{g}}$ is the Riemannian connection on $\mathfrak{g}$ induced by the right–invariant Levi–Civita connection on $G$.
\end{proposition}

\begin{proof}
In the proof of Theorem \ref{orthogonal:input:thm}, we have already seem that $$u^af_a=\text{pr}_{\mathfrak{f}}(\nabla^{\mathfrak{g}}_{\xi} \xi + \grad_{b} V(g)).$$ Applying the dual basis to both sides, we arrive at the desired conclusion.
\end{proof}

Proposition~\ref{orthogonal:input:thm2} provides an explicit feedback law $u^\star$ that acts by projecting the Euler--Poincaré dynamics 
onto the constraint subspace $\mathfrak{k}$ along the input subspace $\mathfrak{f}$, 
so that the motion remains confined to $\mathfrak{k}$ for all time. 

The resulting closed--loop dynamics corresponds to a non-Riemannian geodesic flow on the Lie subgroup $K=\exp(\mathfrak{k})\subset G$ 
with respect to the connection induced by $\nabla^{\mathfrak{c}}$ on which we add the influence of a potential function. Hence, the control $u^\star$ forces the dynamics to evolve in $K$ as if it were 
mechanically constrained to move there.

When the invariant subspace corresponds to a Lie subalgebra, the enforced virtual holonomic constraint restricts the motion to a Lie subgroup, yielding structured motion primitives such as rotations about fixed axes or motions along symmetry subgroups.

The following example illustrates the application of Theorem~\ref{orthogonal:input:thm2} 
to a quadrotor UAV modeled on the Lie group $\SE(3)$.

\subsection{Virtual Holonomic Constraint for a Quadrotor on $\SE(3)$}\label{UAV holonomic example}

We consider a quadrotor evolving on $\SE(3)=\SO(3)\ltimes\mathbb{R}^3$, with configurations denoted by $g=(R,r)$, where $R\in\SO(3)$ is the attitude and $r\in\mathbb{R}^3$ the position.  
The left-trivialized velocity is $\xi := g^{-1}\dot g = (\omega,v)\in\mathfrak{se}(3)$, with $\omega\in\mathbb{R}^3$ the body angular velocity, defined by $\dot R = R\widehat{\omega}$, where $\widehat{\omega}$ is the hat map giving the matrix in $\mathfrak{so}(3)$ uniquely determined by $\omega$, and $v\in\mathbb{R}^3$ the body linear velocity, defined by $\dot r = R v$.

The left-invariant metric is
\[
\langle (\omega,v),(\omega',v')\rangle 
  = \omega^\top J \omega' + m\, v^\top v',
\,\,
J=\mathrm{diag}(J_1,J_2,J_3).
\]

The gravitational potential is $V(R,r)=m g_c\, e_3^\top r$, where $g_c \in \R$ is the acceleration of gravity, and the corresponding force of gravity in body--frame is 
$\text{grad}_{b} V(R,r)=(0,R^{\top}(0,0,mg_c)).$ For notational purposes we will denote by $\grad \; V (r) = (0,0,mg_c)$ the component of the gravitational force  on $\R^{3}$.

The forced Euler--Poincar\'e equation on $\mathfrak{se}(3)$ is
\[
\dot\xi+\nabla^{\mathfrak g}_\xi\xi = -\text{grad}_{b} V(R,r),
\]

Consider the basis $\operatorname{span}\{X_i\}_{i=1}^{6}$ for $\mathfrak{se}(3)$ where $X_1=(\widehat e_1,0)$, $X_2=(\widehat e_2,0)$, $X_3=(\widehat e_3,0)$, $X_4=(0,e_1)$, $X_5=(0,e_2)$, and $X_6=(0,e_3)$.

Let the control input subspace be $\mathfrak f=\operatorname{span}\{X_1,X_2,X_6\}$, so that 
the controlled system is
\begin{equation}\label{eq:hol-quad-left}
\dot\xi + \nabla^{\mathfrak g}_\xi\xi = -\text{grad}_{b} V(R,r)
+u^1 X_1+u^2 X_2+u^3 X_6,\,\dot g = (L_g)_\ast(\xi).
\end{equation}

Using the expression of the Levi--Civita connection for left-invariant metrics on $\SE(3)$,
\[
\nabla^{\mathfrak g}_\xi\xi
=\big(-J^{-1}(\omega\times J\omega),\; -\, v\times \omega\big),
\]
we obtain the explicit reduced dynamics:
\[
J\dot\omega = \omega\times(J\omega) + u_1 e_1 + u_2 e_2,
\]
\[
\dot v = v\times\omega - R^{\top}\grad \; V(r) + u_3 e_3.
\]

We impose the configuration constraint $Re_3=e_3$ with $r=(x,y,z_d)$, which defines the Lie subgroup
\[
K=\Big\{
(R,r)\in \SE(3)\;\Big|\;
R e_3=e_3,\;\ r=(x,y,z_d)
\Big\}.
\]
Its Lie algebra is $\mathfrak{k}
=\operatorname{span}\{X_3,X_4,X_5\}
\subset\mathfrak{se}(3)$ and it is orthogonal to $\mathfrak f$ under the kinetic metric.
Let $\text{pr}_{\mathfrak k}:\mathfrak g\to\mathfrak k$ denote the orthogonal projection.

By Proposition \ref{orthogonal:input:thm2}, the constraint is enforced by the control law
\[
u^\star = -\,\text{pr}_{\mathfrak f} \left[ \Big(-J^{-1}(\omega\times J\omega),\; -v\times\omega + R^{T}\grad \; V(r)\Big) \right],
\]
which gives
\begin{equation*}
    \begin{split}
        u^\star = - & (-J^{-1}(\omega\times J\omega))^1 X_1 + (-J^{-1}(\omega\times J\omega))^2 X_2 + \\
        & \left((-v\times\omega + R^{\top}\grad \; V(r))^3\right)X_6,
    \end{split}
\end{equation*}
where the superscripts are just the coordinates of the involved vectors with respect to the canonical basis $\{e_i\}$ of $\R^{3}$.

Thus the enforcing feedback is
\[
u^{\star \; 1} = \big(-J^{-1}(\omega\times J\omega)\big)^1,
\qquad
u^{\star \; 2} = \big(-J^{-1}(\omega\times J\omega)\big)^2,
\]
\[
u^{\star \; 3} = \big(-v\times\omega + R^{\top}\grad \; V(r)\big)^3.
\]

If $J=\mathrm{diag}(J_1,J_2,J_3)$, then
\[
u^{\star \; 1} = -\frac{J_3-J_2}{J_1}\,\omega_2\omega_3,
\qquad
u^{\star \; 2} = -\frac{J_1-J_3}{J_2}\,\omega_1\omega_3,
\]
\[
u^{\star \; 3} = -(v_x\omega_2 - v_y\omega_1) - (R^{\top} \grad \; V(r))^3.
\]

\medskip

The feedback cancels the $\mathfrak f$--component of the drift 
 (including gravity), making $\mathfrak k$ invariant under the closed loop.  
Thus, if $(g(0),\xi(0))\in K\times\mathfrak k$, the quadrotor maintains vertical 
attitude and constant altitude, with the thrust input tracking the gravitational term 
$(R^{\top} \grad \; V(r))^3$.

In the context of this paper, the enforced virtual constraint defines an invariant low--dimensional manifold of the closed--loop dynamics. The motion evolving on this manifold constitutes a \emph{motion primitive}, understood as a geometrically characterized dynamical regime rather than a single trajectory or equilibrium point. Different choices of virtual constraints correspond to different motion primitives, such as rotations about fixed axes, planar motions, or periodic regimes, which can be selected and combined at the level of control design.

\section{Virtual Nonholonomic Constraints on Lie Groups}\label{sec5}
Before introducing the concept of virtual nonholonomic constraints on Lie groups, we shall develop some results on the dynamics of nonholonomic systems on Lie groups. 


\subsection{Nonholonomic mechanical systems}
Let $Q$ be the configuration space of a mechanical system, a differentiable manifold with $\dim(Q)=n$, and with local coordinates denoted by $(q^i)$ for $i=1,\ldots,n$. Most nonholonomic systems have linear constraints on velocities, and these are the ones we will consider in this subsection.

A nonholonomic constraint linear in the velocities can be written in Pfaffian form as
\begin{equation}\label{NHconstraint}
\varphi^a(q,\dot q) = \mu^a_i(q)\,\dot q^i = 0, \qquad a=1,\dots,m,\end{equation}where $\mu^a(q) = \mu^a_i(q)\,dq^i$ are independent one–forms on $Q$. 
The constraint distribution and its annihilator are then
\begin{align*}\mathcal{D}_q &= \{\dot q \in T_q Q \mid \mu^a(q)(\dot q)=0,\ \forall a\},\\ 
\mathcal{D}_q^0 &= \mathrm{span}\{\mu^a(q)\}_{a=1}^m \subset T^*_qQ.\end{align*}

Next consider mechanical systems with a Lagrangian of mechanical type, $L(v_q)= \tfrac12 \|v_q\|^2 - V(q)$ with $v_q\in T_qQ$, subject to the nonholonomic constraints \eqref{NHconstraint}.  

A \textit{nonholonomic mechanical system} on a smooth manifold $Q$ is given
by the pair $(\left< \cdot, \cdot\right>, \mathcal{D})$, where $\left< \cdot, \cdot\right>$ is
a Riemannian metric on $Q,$ representing the kinetic energy of the
system and $\mathcal{D}$ is a regular distribution on $Q$
describing the nonholonomic constraints. Denote by $\tau_{\mathcal{D}}:\mathcal{D}\rightarrow Q$ the canonical
projection from $\mathcal{D}$ to $Q$, locally given by $\tau_{\mathcal{D}}(q^i, \dot{q}^i)=q^i$, and denote by 
$\Gamma(\tau_{\mathcal{D}})$ the set of sections of $\tau_{D}$, that is, $Z\in\Gamma(\tau_{\mathcal{D}})$ if $Z:Q\to\mathcal{D}$ satisfies $(\tau_{\mathcal{D}}\circ Z)(q)=q$. We also denote by $\mathfrak{X}(Q)$ the set of vector fields on $Q$. If $X, Y\in\mathfrak{X}(Q),$ then
$[X,Y]$ denotes the standard Lie bracket of vector fields.

The trajectories $q:I\rightarrow Q$ of a mechanical Lagrangian determined by a kinetic Lagrangian function satisfy the following equation
\begin{equation}\label{ELeq}
    \nabla_{\dot{q}}\dot{q}= 0
\end{equation} which are just the geodesics with respect to the connection $\nabla$. 

Using the Riemannian metric we can define two
complementary orthogonal projectors ${\mathcal P}\colon TQ\to {\mathcal D}$ and ${\mathcal Q}\colon TQ\to {\mathcal
D}^{\perp},$ with respect to the tangent bundle orthogonal decomposition $\mathcal{D}\oplus\mathcal{D}^{\perp}=TQ$. In the presence of a constraint distribution $\mathcal{D}$, equation \eqref{ELeq} must be slightly modified as follows. Consider the \textit{nonholonomic connection} $\nabla^{nh}:\mathfrak{X}(Q)\times \mathfrak{X}(Q) \rightarrow \mathfrak{X}(Q)$ defined by (see \cite{bullo} for instance)
\begin{equation}\label{nh:connection}
    \nabla^{nh}_X Y =\nabla_{X} Y + (\nabla_{X} \mathcal{Q})(Y).
\end{equation} where $\nabla$  is the Levi-Civita connection on $Q$.

Thus, the constrained dynamics associated with the mechanical Lagrangian with potential 
satisfies
\begin{equation}\label{nonholonomic:mechanical:equation}
    \nabla^{nh}_{\dot q}\dot q + \mathcal{P}\big(\mathrm{grad}\,V\big) = 0,
\qquad \dot q(t)\in\mathcal{D}_{q(t)}.
\end{equation}

Equation~\eqref{nonholonomic:mechanical:equation} is understood as the system being subject to  the constraint 
$\dot q(t)\in \D_{q(t)}$ for all $t$, 
so that the admissible velocities always belong to the distribution $\D$. 
In this setting, the reaction forces act in the orthogonal complement $\D^\perp$, 
and the projection $\mathcal{P}:TQ\to \D$ ensures that only the admissible component 
of the acceleration is retained in the dynamics. 
Equivalently, the nonholonomic equations can be interpreted as the orthogonal 
projection of the unconstrained Euler--Lagrange equation~\eqref{ELeq} onto the 
constraint distribution $\D$.

 We have the following results relating geodesics with respect to both connections.

\begin{lemma}[\cite{bullo}, Section 2]\label{Lewis:lemma}
    Given a Riemannian manifold $Q$, letting $\nabla$ be the Levi-Civita connection and $\D$ a non-integrable distribution then a curve $q:[a,b]\to Q$ is a geodesic of the nonholonomic connection $\nabla^{nh}$, if and only if
    $$\nabla_{\dot{q}}\dot{q} \in \D^{\bot} \text{ and } \dot{q} \in \D.
    $$
\end{lemma}
We refer to those geodesics as \textit{nonholonomic trajectories}. 

When a potential $V(q)$ is present, the actual trajectories of the 
nonholonomic mechanical system satisfy the forced equation 
$\nabla^{nh}_{\dot q}\dot q + \mathcal{P}(\mathrm{grad}\,V)=0$ 
and therefore need not be nonholonomic geodesics in the sense of 
Lemma~\ref{Lewis:lemma}.

\subsection{Nonholonomic mechanics on Lie groups}

Consider a left--invariant distribution $\D$ on a Lie group $G$, that is, for each 
$g\in G$ the fiber $\D_g$ is defined by $\D_g = T_e L_g(\h)$, where $\h$ is a fixed 
subspace of the Lie algebra $\g$.  
Using the inner product $\langle\cdot,\cdot\rangle_{\g}$ on $\g$ define the 
orthogonal complement of $\h$ by $\h^\perp=\{\xi\in\g\;:\;\langle\xi,\eta\rangle_{\g}=0,\ \forall\eta\in\h\}$, so that $\g=\h\oplus \h^\perp$.  
We denote the associated orthogonal projectors by 
$\mathfrak{P}:\g\to\h$ and $\mathfrak{Q}:\g\to\h^\perp$ and we assume that the 
Riemannian metric on $G$ is left--invariant.  
Then both $\D$ and $\D^\perp$ are left--invariant subbundles.

\begin{lemma}
Let $\D$ be a left--invariant distribution on $G$.  
Given a left--invariant metric on $G$, let $\D^\perp$ be its associated 
orthogonal distribution and let $\mathcal{P}:TG\to\D$, $\mathcal{Q}:TG\to\D^\perp$ 
be the orthogonal projectors. Then:
\begin{enumerate}
    \item The orthogonal distribution $\D^\perp$ is left--invariant and 
    \[
        \D^\perp_g = T_e L_g(\h^\perp).
    \]
    \item The Lie algebra projectors satisfy
    \begin{equation}\label{Lie algebra projections-LI}
        \mathfrak{P}
        = T_g L_{g^{-1}} \circ \mathcal{P} \circ T_e L_g,
        \,\,
        \mathfrak{Q}
        = T_g L_{g^{-1}} \circ \mathcal{Q} \circ T_e L_g.
    \end{equation}
\end{enumerate}
\end{lemma}

\begin{proof}
1. Let $X,Y$ be vector fields such that $X\in\D$ and $Y\in\D^\perp$.  
Left--invariance of the metric implies
\[
\langle X_g , Y_g\rangle
=
\langle T_g L_{g^{-1}} X_g , T_g L_{g^{-1}} Y_g\rangle_{\g}.
\]
Since $X_g\in T_e L_g(\h)$, we have $T_g L_{g^{-1}} X_g\in\h$.  
Thus orthogonality at $g$ implies $T_gL_{g^{-1}}(Y_g)\in\h^\perp$, hence 
$Y_g = T_e L_g(\eta)$ for some $\eta\in\h^\perp$.  
Thus $\D^\perp$ is left--invariant and $\D^\perp_g=T_eL_g(\h^\perp)$.

2. Let $\xi\in\g$ and write $\xi = \xi^\top + \xi^\perp$ according to 
$\g=\h\oplus\h^\perp$.  
Left translation gives $T_e L_g(\xi) = \xi_L = \xi^\top_L + \xi^\perp_L$. 
Since $\D_g = T_e L_g(\h)$ and $\D^\perp_g=T_e L_g(\h^\perp)$, we have
$\mathcal{P}(T_e L_g(\xi)) = \xi^\top_L$ and 
$\mathcal{Q}(T_e L_g(\xi)) = \xi^\perp_L$.  
Applying $T_g L_{g^{-1}}$ gives 
$T_g L_{g^{-1}}(\xi^\top_L) = \xi^\top = \mathfrak{P}(\xi)$ and 
$T_g L_{g^{-1}}(\xi^\perp_L) = \xi^\perp = \mathfrak{Q}(\xi)$.
\end{proof}

Next we define the \emph{nonholonomic $\h$--connection} on the Lie algebra.
Let $\nabla^{nh}$ be the nonholonomic connection associated with the 
Levi--Civita connection $\nabla$ on $(G,\langle\cdot,\cdot\rangle)$.  
For $\xi,\eta\in\g$ let $\xi_L,\eta_L$ denote the corresponding left--invariant 
vector fields.  
We define:
\begin{equation}\label{h-connection-LI}
    \nabla^\h_{\xi}\eta 
    := \big(\nabla^{nh}_{\xi_L}\eta_L\big)(e).
\end{equation}

\begin{proposition}
Let $\mathfrak{P}:\g\to\h$ and $\mathfrak{Q}:\g\to\h^\perp$ be the 
orthogonal projectors associated with the left--invariant metric.  
Define the bilinear map
\[
(\nabla^{\g}_{\xi}\mathfrak{Q})(\eta)
:= \nabla^{\g}_{\xi}(\mathfrak{Q}(\eta))
   - \mathfrak{Q}(\nabla^{\g}_{\xi}\eta),
\]
where $\nabla^\g$ is the Levi--Civita connection induced on the Lie algebra.  
Then for all $\xi,\eta\in\h$ we have
\begin{equation}\label{eqq1-left}
    \nabla^{\h}_{\xi}\eta 
    = \mathfrak{P}\big(\nabla^{\g}_{\xi}\eta\big).
\end{equation}
\end{proposition}

\begin{proof}
From the definition~\eqref{h-connection-LI}, $\nabla^\h_{\xi}\eta 
= (\nabla^{nh}_{\xi_L}\eta_L)(e)
= \big(\nabla_{\xi_L}\eta_L + (\nabla_{\xi_L}\mathcal{Q})(\eta_L)\big)(e)$. Using the identities in Lemma~$3$ and evaluating at 
$e\in G$ gives the desired expression.
\end{proof}

Thus, for $\xi,\eta\in\h$ we obtain the explicit form
\begin{equation}\label{hcon_decomp}
    \nabla^{\h}_{\xi} \eta 
    =\frac12 \mathfrak{P}\big([\xi, \eta]_\g 
       -\sharp\left[\ad_\xi \flat(\eta)\right]
    -\sharp \left[\ad_\eta \flat(\xi)\right]\big).
\end{equation}

\begin{lemma}\label{lemma: Horcov-to-covh-left}
Let $g: [a,b] \to G$ be a curve and $X$ a smooth vector field along $g$ 
satisfying $X(t)\in \D_{g(t)}$. 
Suppose that $\xi(t) = g(t)^{-1}\dot{g}(t)$ and $\eta(t) = g(t)^{-1} X(t)$. Then the following relation holds for all $t \in [a,b]$:
\begin{equation}\label{eq:Horcov-to-covh-left}
    \nabla^{nh}_{\dot{g}} X(t) 
    = g(t)\big(\dot{\eta}(t) + \nabla_{\xi}^{\h} \eta(t) \big).
\end{equation}
\end{lemma}


\begin{theorem}
Suppose that $g: [a,b] \to G$ is a nonholonomic trajectory with respect to a 
left--invariant metric and distribution $\D$, and let 
$\xi(t) = g(t)^{-1}\dot{g}(t)$.  
Then $\xi$ satisfies the reduced equation
\begin{equation}\label{reduced-trajectory-left-V}
    \dot{\xi}(t) + \nabla^{\h}_{\xi(t)}\xi(t)
    + \mathfrak{P}\Big(
       \mathrm{grad}_b\,V(g(t))  \Big)
    = 0,
\end{equation}
or, equivalently,
\begin{equation}\label{reduced-trajectory-left-V-dual}
    \dot{\xi}(t) 
    + (\mathfrak{P} \circ \sharp)
      \big[ \operatorname{ad}^{*}_{\xi(t)}\flat(\xi(t))\big]
    + (\mathfrak{P}\circ\sharp)\big(\mu_V(g(t))\big)
    = 0,
    \, \xi(t)\in \h,
\end{equation}
where $\mu_V(g(t)) 
:= L_{g(t)}^{*}\big(dV(g(t))\big)\in\g^{*}$ is the left--trivialized covector associated with the potential $V$ determined by the pullback of $dV$ to the identity.
\end{theorem}

\begin{proof}
A nonholonomic trajectory with potential satisfies
\[
\nabla^{nh}_{\dot g}\dot g + \mathcal{P}\big(\mathrm{grad}\,V(g)\big)=0.
\]
Applying Lemma~\ref{lemma: Horcov-to-covh-left} to $X(t)=\dot g(t)$ and using 
$\eta(t)=\xi(t)$ yields
\[
g(t)\big(\dot\xi(t)+\nabla^\h_{\xi(t)}\xi(t)\big)
+ \mathcal{P}\big(\mathrm{grad}\,V(g(t))\big)=0.
\]
Applying $T_{g(t)}L_{g(t)^{-1}}$ and using 
$T_gL_{g^{-1}}\circ\mathcal{P}=\mathfrak{P}\circ T_gL_{g^{-1}}$ gives the reduced equation~\eqref{reduced-trajectory-left-V}.

For the dual form, use the explicit expression of the $\h$--connection and 
define the left--trivialized potential momentum 
$\mu_V(g)=L_{g(t)}^{*}(dV(g))$.  
Left--invariance of the metric implies $ T_gL_{g^{-1}}\big(\mathrm{grad}\,V(g)\big)=\sharp(\mu_V(g))$.
\end{proof}

\subsection{Virtual nonholonomic constraints on Lie groups}\label{sec5.1}

Similarly to Section~\ref{vclg}, we consider a controlled mechanical system on the Lie group $G$ 
with configuration $g\in G$ and body velocity $\xi\in\mathfrak{g}$, defined by the left trivialization $\xi = g^{-1}\dot g$. Let the control force $F:\mathfrak{g}\times U \rightarrow \mathfrak{g}^{*}$ be given by
\begin{equation}
    F(\xi,u) = \sum_{a=1}^{m} u^{a}F^{a}(\xi),
\end{equation}
where $F^{a}\in\mathfrak{g}^{*}$, $m<n$, $U\subset\mathbb{R}^{m}$ is the control set, 
and $u^{a}\in\mathbb{R}$, $1\leq a\leq m$, are the control inputs. 
Let $f_a = \sharp(F^a|_{\g})\in\g$ be the corresponding controlled directions and define the input subspace $\mathfrak{f} = \mathrm{span}\{f_{1},\dots,f_{m}\}\subset\mathfrak{g}$.

For the mechanical Lagrangian $L(g,\dot g)=\tfrac12\|\dot g\|^2 - V(g)$, 
the controlled Euler--Poincar\'e equations in left-trivialized form read
\begin{equation}\label{nh-control:system}
    \dot{g} = T_{e}L_{g}(\xi),\,
    \dot{\xi} 
    + \nabla^{\mathfrak{g}}_{\xi}\xi 
    + \grad_{b} V(g)
    = u^{a}f_{a}.
\end{equation}
\begin{definition}
   A \emph{virtual nonholonomic constraint} associated with the mechanical system 
   \eqref{nh-control:system} is a controlled invariant subspace $\mathfrak{d}\subset\mathfrak{g}$, 
   that is, there exists a control law $u^\star$ such that the closed--loop system 
   satisfies: if $\xi(0)\in\mathfrak{d}$ then $\xi(t)\in\mathfrak{d}$ for all $t\geq 0$.
\end{definition}

Provided that $\mathfrak{d}$ and $\mathfrak{f}$ are transversal there exists a unique control law $u^{\star}$ making $\mathfrak{d}$ control invariant due to Theorem \ref{maintheorem}.

\begin{remark}
A linear subspace $\mathfrak{d}\subset\g$ induces a left--invariant distribution 
$\D_g = T_e L_g(\mathfrak{d})$ on $G$.  
If $\mathfrak{d}$ is a Lie subalgebra, the distribution $\D$ is integrable and 
its integral manifolds are left cosets of the Lie subgroup $H=\exp(\mathfrak{d})$.
In that case, if the closed--loop dynamics keep $\xi(t)$ inside $\mathfrak{d}$, then 
$\dot g(t)=T_eL_{g(t)}(\xi(t))\in\D_{g(t)}$, and any trajectory $g(t)$ starting in $H$ remains in $H$ for all time.  
Thus, in the integrable case, a virtual nonholonomic constraint behaves effectively as a 
virtual holonomic constraint on the subgroup $H\subset G$.
\end{remark}

We now focus on the orthogonal situation, which parallels the holonomic case in 
Theorem~\ref{thm:holonomic-EP-left}. Let $\h\subset\g$ be a fixed subspace and consider the orthogonal decomposition $\g = \h \oplus \h^\perp$, with respect to the inner product induced by the left--invariant metric.  
Let $\mathfrak{P}:\g\to\h$ and $\mathfrak{Q}:\g\to\h^\perp$ denote the corresponding 
orthogonal projectors, and let $\nabla^\h$ be the nonholonomic $\h$--connection 
introduced in the previous subsection, so that for $\xi,\eta\in\h$,
\[
\nabla^\h_{\xi}\eta = \mathfrak{P}\big(\nabla^{\g}_{\xi}\eta\big).
\]

\begin{theorem}\label{nh-orthogonal:input:thm}
Assume that the control input subspace coincides with the orthogonal complement, $\mathfrak{f} = \h^\perp = \mathrm{span}\{f_1,\dots,f_m\}$, and that $\{f_a\}$ is a basis of $\h^\perp$.

Then, for any initial condition with $\xi(0)\in\h$, there exists a unique feedback 
control law $u^\star(\xi,g)$ that renders $\h$ invariant under the closed--loop 
dynamics of \eqref{nh-control:system}.  

Moreover, this feedback is given implicitly by
\begin{equation}\label{eq:nh-feedback-implicit}
    u^{\star}_a f_a 
    = \mathfrak{Q}\Big(
        \nabla^{\g}_{\xi}\xi 
        + \grad_{b} V(g)
      \Big)\in\h^\perp,
\end{equation}
and the reduced closed--loop dynamics on $\h$ are
\begin{equation}\label{eq:nh-reduced-constraint}
    \dot{\xi} 
    + \nabla^{\h}_{\xi}\xi
    + \mathfrak{P}\Big(
        \grad_{b} V(g)
      \Big) = 0,
    \qquad \xi(t)\in\h.
\end{equation}
\end{theorem}

\begin{proof}
Rewrite \eqref{nh-control:system} as
\[
\dot{\xi} 
= -\,\nabla^{\g}_{\xi}\xi 
  - \grad_{b} V(g)
  + u^{a}f_a.
\]
We seek a feedback such that, for $\xi\in\h$, the reduced dynamics on $\h$ satisfy
\[
\dot{\xi} + \nabla^\h_{\xi}\xi
+ \mathfrak{P}\Big(
    \grad_{b} V(g)
  \Big) = 0.
\]
Using $\nabla^\h_{\xi}\xi = \mathfrak{P}(\nabla^\g_{\xi}\xi)$, this is equivalent to
\[
\dot{\xi} 
+ \mathfrak{P}\Big(
    \nabla^\g_{\xi}\xi 
    + \grad_{b} V(g)
  \Big) = 0,
\qquad \xi\in\h.
\]
Substituting the controlled dynamics for $\dot\xi$ gives
\begin{align}
&-\nabla^{\g}_{\xi}\xi 
 - \grad_{b} V(g)
 + u^{a}f_a +\, \mathfrak{P}\Big(
    \nabla^\g_{\xi}\xi 
    + \grad_{b} V(g)
  \Big) = 0,
\end{align}
hence
\begin{align}
u^{a}f_a
&= \big(I-\mathfrak{P}\big)\Big(
    \nabla^{\g}_{\xi}\xi 
    + \grad_{b} V(g)
  \Big) \nonumber\\
&= \mathfrak{Q}\Big(
    \nabla^{\g}_{\xi}\xi 
    + \grad_{b} V(g)
  \Big).
\end{align}

The right--hand side lies in $\h^\perp$, and since $\{f_a\}$ is a basis of $\h^\perp$, 
this uniquely determines the coefficients $u^{\star}_a$, proving \eqref{eq:nh-feedback-implicit} 
and the existence and uniqueness of the feedback.

Substituting $u^\star$ back into the controlled dynamics yields
\[
\dot{\xi} 
+ \nabla^{\g}_{\xi}\xi 
+ \grad_{b} V(g)
= \mathfrak{Q}\Big(
    \nabla^{\g}_{\xi}\xi 
    + \grad_{b} V(g)
  \Big),
\]
that is,
\[
\dot{\xi} 
+ \mathfrak{P}\Big(
    \nabla^{\g}_{\xi}\xi 
    + \grad_{b} V(g)
  \Big) = 0.
\]
For $\xi\in\h$, this is exactly \eqref{eq:nh-reduced-constraint}.
Since the right--hand side belongs to $\h$, solutions starting in $\h$ remain 
in $\h$, so $\h$ is invariant under the closed--loop flow.
\end{proof}

\medskip

Theorem~\ref{nh-orthogonal:input:thm} shows that, in the orthogonal case 
$\g=\h\oplus\h^\perp$, the control inputs cancel the $\h^\perp$--component of the 
Euler--Poincar\'e drift $B(\xi,g)$ and project the dynamics onto $\h$, where they evolve according to the nonholonomic connection $\nabla^\h$ plus the projected potential term.


\begin{remark}
The transversality condition $\mathfrak g=\mathfrak d\oplus\mathfrak f$, and in particular the orthogonal case 
$\mathfrak f=\mathfrak d^\perp$ with respect to the kinetic inner product, plays a distinguished role in the setting of 
virtual nonholonomic constraints. In the orthogonal case, cancelling the $\mathfrak f$--component of the 
Euler--Poincar\'e drift corresponds to an orthogonal (d'Alembert--type) projection of the dynamics, yielding closed--loop 
equations on $\mathfrak d$ governed by the nonholonomic $\mathfrak d$--connection 
$\nabla^{\mathfrak d}=\mathfrak{P}\circ\nabla^{\mathfrak g}$.

This structure is closely related to classical Cartan--type decompositions of semi--simple Lie algebras, where 
$\mathfrak g=\mathfrak k\oplus\mathfrak p$ with $\mathfrak p=\mathfrak k^\perp$ under a bi--invariant inner product 
\cite{helgason1979differential}. In this analogy, $\mathfrak d$ plays the role of the constraint directions (analogous to 
$\mathfrak k$), while the input space $\mathfrak f$ corresponds to the complementary directions (analogous to 
$\mathfrak p$). The enforcing feedback cancels precisely the $\mathfrak f$--component of the Euler--Poincar\'e drift, 
projecting the dynamics onto $\mathfrak d$ and producing a constrained evolution reminiscent of symmetric--space 
reduction \cite{stratoglou2024nonholonomic}.

A classical example arises in $G=\mathrm{SO}(3)$ with Lie algebra $\mathfrak g=\mathfrak{so}(3)$. Fix a unit vector 
$e_3\in\mathbb{R}^3$ and define 
$\mathfrak d:=\mathrm{span}\{\widehat{e_3}\}$ and 
$\mathfrak f:=\mathrm{span}\{\widehat{e_1},\widehat{e_2}\}$. These subspaces satisfy the Cartan--type relations 
$[\mathfrak d,\mathfrak d]\subseteq\mathfrak d$, 
$[\mathfrak d,\mathfrak f]\subseteq\mathfrak f$, and 
$[\mathfrak f,\mathfrak f]\subseteq\mathfrak d$, and yield an orthogonal decomposition 
$\mathfrak g=\mathfrak d\oplus\mathfrak f$ under the Killing form \cite{bloch2003nonholonomic}. 

Within the virtual nonholonomic constraint framework, selecting $\mathfrak d$ as the constraint subspace and 
$\mathfrak f$ as the control input space admits a clear geometric interpretation: the control inputs act along directions 
transverse to the symmetry axis, while the feedback enforces invariance of the dynamics on $\mathfrak d$. Under the 
unique feedback law of Theorem~\ref{maintheorem}, the $\mathfrak f$--component of the Euler--Poincar\'e drift is cancelled 
and the closed--loop dynamics evolve on $\mathfrak d$, recovering the Cartan--type splitting 
$\mathfrak g=\mathfrak d\oplus\mathfrak f$ at the level of the reduced dynamics.
\hfill$\diamond$
\end{remark}

\begin{figure}[h!]
\centering

\tdplotsetmaincoords{70}{110}

\begin{tikzpicture}[scale=2.1,tdplot_main_coords]

  \shade[ball color=gray!15,opacity=0.95] (0,0,0) circle (1);

  \draw[->,thick] (0,0,0) -- (1.3,0,0) node[below] {$e_1$};
  \draw[->,thick] (0,0,0) -- (0,1.3,0) node[left] {$e_2$};
  \draw[->,very thick,blue] (0,0,0) -- (0,0,1.3) node[above] {$e_3$};


  \draw[thick,red] plot[domain=0:360,samples=120,variable=\t]
      ({cos(\t)},{sin(\t)},0);

  \draw[thick,green!60!black] plot[domain=-90:90,samples=120,variable=\t]
      ({cos(\t)},0,{sin(\t)});

  \draw[thick,green!60!black,dashed] plot[domain=-90:90,samples=120,variable=\t]
      (0,{cos(\t)},{sin(\t)});

  \draw[thick,orange!80!black]
    plot[domain=0:360,samples=120,variable=\t]
      ({0.8*cos(\t)},{0.8*sin(\t)},{sqrt(1-0.8*0.8)});

  \coordinate (Xi) at (0.55,0.35,0.75);
  \draw[->,thick,black] (0,0,0) -- (Xi) node[above right] {$\xi$};

  \coordinate (Xi_h) at (0,0,0.75);
  \draw[dashed,blue] (Xi) -- (Xi_h);
  \fill[blue] (Xi_h) circle (0.7pt) node[left] {$\mathfrak{P}(\xi)$};

  \coordinate (Xi_f) at (0.55,0.35,0);
  \draw[dashed,gray!60] (Xi) -- (Xi_f);
  \fill[gray!60] (Xi_f) circle (0.7pt) node[below right] {$\mathfrak{Q}(\xi)$};

\end{tikzpicture}

\caption{
Geometric visualization of the orthogonal decomposition 
$\mathfrak{so}(3) = \mathfrak d \oplus \mathfrak f$ on the unit sphere.
The axis $\mathfrak d = \mathrm{span}\{e_3\}$ (blue) represents the constraint 
subspace, while the plane $\mathfrak f = \mathrm{span}\{e_1,e_2\}$ corresponds to 
its orthogonal complement.
The equator (red) and meridians (green) illustrate geodesic curves associated 
with the induced metric.
A vector $\xi$ decomposes as $\xi = \mathfrak P(\xi) + \mathfrak Q(\xi)$ via the 
orthogonal projectors onto $\mathfrak d$ and $\mathfrak f$, reflecting the 
projection structure underlying virtual nonholonomic constraints.}

\label{fig:cartan-geodesic}
\end{figure}

\subsection{Virtual nonholonomic constraint for a quadrotor UAV}

We continue the quadrotor example on $SE(3)$
from the previous subsection. The uncontrolled
Euler--Poincar\'e equations take the form
\[
\dot\xi+\nabla^{\mathfrak g}_{\xi}\xi 
+ \text{grad}_{b} V(R,r)= 0.
\]
Recall that for $\xi=(\omega,v)$ we have the explicit expression
\[
\nabla^{\mathfrak g}_\xi\xi
=\big(-J^{-1}(\omega\times J\omega),\; -v\times\omega\big).
\]

We now add control inputs acting as \emph{effective} body--frame forces
along $e_2$ and $e_3$, which in practice can be realized by suitable
combinations of thrust and fast attitude regulation.  With respect to the
orthonormal basis $\{X_{1},X_{2},X_{3},X_{4},X_{5},X_{6}\}$ defined previously, we take now $\mathfrak f = \text{span} \{X_5,X_6\}$, so that the controlled Euler--Poincar\'e equations read
\begin{equation}\label{eq:se3-quad-controlled-nh}
\dot\xi+\nabla^{\mathfrak g}_{\xi}\xi
+ \text{grad}_{b} V(R,r)
= u^1X_5+u^2X_6.
\end{equation}
In components,
\[
\dot v = v\times\omega - R^{\top}\grad \; V(r) + u^1 e_2 + u^2 e_3,
\,\,
J\dot\omega = \omega\times (J\omega).
\]

Inspired by standard navigation primitives for quadrotors (altitude--hold and
no--sideslip motion), we impose the linear velocity constraints
\begin{equation}\label{eq:quad-nh-constraints-new}
v_y = 0,
\qquad
v_z = 0,
\end{equation}
which enforce zero vertical velocity and zero lateral velocity in the body
frame.  These constraints define the linear subspace
\[
\mathfrak d
=\{\xi=(\omega,v)\in\mathfrak{se}(3)\mid v=(v_x,0,0)\},
\]
so that $\omega=(\omega_1,\omega_2,\omega_3)$ is arbitrary and only the
forward body--frame velocity $v_x$ is allowed.  Equivalently, $\mathfrak d=\operatorname{span}\{X_1,X_2,X_3,X_4\}$. A short computation using the Lie bracket on $\mathfrak{se}(3)$ shows that
$[X_2,X_4]=X_6\notin\mathfrak d$, so $\mathfrak d$ is not a Lie subalgebra and the associated left--invariant
distribution is genuinely nonholonomic.

Geometrically, $\mathfrak d$ captures \emph{planar navigation} with
altitude hold and no sideslip: the vehicle is allowed to yaw and tilt, and
to move forward in its body $e_1$--direction, but it does not drift
sideways ($v_y=0$) or along the vertical in the body frame ($v_z=0$).
The complementary control subspace $\mathfrak f = \operatorname{span}\{X_5,X_6\}$ acts precisely along the forbidden directions (lateral and vertical body
forces).  Thus we have an orthogonal decomposition
$\mathfrak g = \mathfrak d \oplus \mathfrak f$, where $\mathfrak d$ encodes the desired ``navigation'' behaviour and
$\mathfrak f$ provides the directions we use to enforce the virtual
nonholonomic constraint.

Let $\mathfrak P:\mathfrak g\to\mathfrak d$ and
$\mathfrak Q:\mathfrak g\to\mathfrak f$ be the associated orthogonal
projectors.  The drift of~\eqref{eq:se3-quad-controlled-nh} is
\[
B(\xi,g)
= \big(-J^{-1}(\omega\times J\omega),\; -v\times\omega + R^{\top}\grad \; V(r)\big).
\]
According to the virtual constraint construction in
Section~\ref{sec5}, the feedback $u^\star$ enforcing the constraint
$\xi\in\mathfrak d$ is uniquely determined by cancelling the
$\mathfrak f$--component of the drift:
\[
u^{\star}_1X_5+u^{\star}_2X_6
= \mathfrak Q\big(B(\xi,g)\big),
\]
or, using the dual basis $\{F^1,F^2\}$ of $\{X_5,X_6\}$, $u^{\star}_a = F^a\!\left(\mathfrak Q\big(B(\xi,g)\big)\right)$ for $a=1,2$. Under this control law, if $\xi(0)\in\mathfrak d$, then $\xi(t)\in\mathfrak d,\ \forall t$, and the quadrotor evolves on the virtual nonholonomic constraint subspace $\mathfrak d$.

On $\mathfrak d$, the reduced dynamics is described by the nonholonomic
$\mathfrak d$--connection $\nabla^{\mathfrak d}$, which in this setting is
given by the projection of the full Riemannian $\mathfrak g$--connection,
\[
\nabla^{\mathfrak d}_{\xi}\xi
= \mathfrak P\big(\nabla^{\mathfrak g}_\xi\xi\big),
\qquad \xi\in\mathfrak d.
\]
Substituting the feedback $u^\star$ into
\eqref{eq:se3-quad-controlled-nh} and projecting onto $\mathfrak d$ yields
the reduced closed--loop equation
\begin{equation}\label{eq:reduced-h-quad}
\dot\xi + \nabla^{\mathfrak d}_{\xi}\xi
+ \mathfrak P\big( R^{\top}\grad \; V(r) \big) = 0,
\qquad \xi(t)\in\mathfrak d.
\end{equation}

In terms of the components $(\omega,v_x)$ on $\mathfrak d$, write
$\xi=(\omega,v_x e_1)$.
Using $\nabla^{\mathfrak g}_\xi\xi
=\big(-J^{-1}(\omega\times J\omega), -v\times\omega\big)$ and
$v=(v_x,0,0)$, we obtain
\[
v\times\omega = (0,-v_x\omega_3,\,v_x\omega_2),
\]
whose projection onto $\mathfrak d$ has no translational component along
$e_1$.  Therefore
\[
\nabla^{\mathfrak d}_\xi\xi
=\big(-J^{-1}(\omega\times J\omega),\,0\cdot e_1\big),
\]
and the reduced dynamics~\eqref{eq:reduced-h-quad} becomes
\[
J\dot\omega + \omega\times(J\omega)=0,
\qquad
\dot v_x + (R^{\top}\grad \; V(r))_1 = 0,
\]
with the constraints $v_y=v_z=0$ enforced by the feedback $u^\star$.
In the hovering attitude $R=I$, where $R^{\top}\grad \; V(r)=(0,0,-g)$, the forward
speed $v_x$ is constant and the rotational dynamics reduces to the free
Euler equations on $SO(3)$, while the control inputs $u^\star$ cancel
exactly the lateral and vertical components of gravity and drift so that
the quadrotor remains on the nonholonomic ``navigation'' subspace
$\mathfrak d$.

\subsection{Virtual affine nonholonomic constraints on Lie groups}

We extend the construction of virtual linear nonholonomic constraints to the
affine case, in full parallel with the holonomic and linear nonholonomic
settings.  Throughout this section, we will still consider the control system \eqref{nh-control:system}.

\begin{definition}
A \emph{virtual affine nonholonomic constraint} for
\eqref{nh-control:system} is an affine subspace
$\mathfrak a\subseteq\g$ that is rendered invariant by a feedback control law.
\end{definition}

Let the affine constraint be $\mathfrak a = a_0 + \mathfrak d$, with 
$\mathfrak d$ a vector subspace of $\g$.  As in the linear case, we assume that the control input distribution satisfies $\g = \mathfrak d \oplus \mathfrak f$, which is the transversality condition guaranteeing the unique decomposition of any vector into a constraint component and an input component.

\begin{definition}
Two affine subspaces $W_1,W_2\subseteq V$ are \emph{transversal}, written
$W_1\oplus W_2$, if their model subspaces satisfy
$V = W_{10} \oplus W_{20}$.
\end{definition}

The next theorem merges the existence–uniqueness result and the reduced closed–loop dynamics in a unified form.

\begin{theorem}
\label{thm:affine-unified}
Let $\mathfrak a = a_0 + \mathfrak d$ be an affine subspace of $\g$ and
assume that $\g=\mathfrak d\oplus\mathfrak f$, where 
$\mathfrak f = \mathrm{span}\{f_1,\dots,f_m\}$ is the control input space. Then:

\begin{enumerate}
\item There exists a unique feedback control $u^\star$ that renders
$\mathfrak a$ invariant under the closed--loop dynamics of 
\eqref{nh-control:system}.  

\item  If the decomposition is \emph{orthogonal},
$\mathfrak f = \mathfrak d^\perp$, then
\begin{equation}\label{eq:affine-feedback}
u^\star_a f_a
= 
\mathfrak Q\Big(
    \nabla^{\g}_{\xi}\xi
  + \text{grad}_{b} V(g)
\Big),
\, \mathfrak Q:\g\to\mathfrak f,
\end{equation}
and this law is unique.

\item For any initial condition satisfying $\xi(0)\in\mathfrak a$, the closed--loop
dynamics remains in $\mathfrak a$ for all time and satisfies
\begin{equation}\label{eq:affine-reduced}
\dot{\xi}
+ \nabla^{\mathfrak d}_{\xi-a_0}(\xi-a_0)
+ \mathfrak{P}\!\Big(
    \text{grad}_{b} V(g)
  \Big)
= 0,
\end{equation}
with $\xi(t)\in\mathfrak a$, where $\mathfrak P$ is the orthogonal projector onto $\mathfrak d$ and
$\nabla^{\mathfrak d} = \mathfrak P\circ\nabla^{\g}$ is the induced
nonholonomic $\mathfrak d$–connection.
\end{enumerate}
\end{theorem}

\begin{proof}
A curve remains in the affine set $\mathfrak a = a_0 + \mathfrak d$ iff
$\xi(t)-a_0\in\mathfrak d$, i.e. $\mu^a(\xi(t)-a_0)=0$, for covectors $\mu^a$ spanning the annihilator of $\mathfrak d$.  
Differentiating and substituting \eqref{nh-control:system} yields
\[
\mu^a(\dot\xi)
= -\,\mu^a(B(\xi,g))
  + u^b\mu^a(f_b).
\]
Transversality implies that $\mu^a(f_b)$ is invertible, hence the invariance
conditions uniquely determine the feedback, proving (1). In the orthogonal case $\mathfrak f=\mathfrak d^\perp$, the drift decomposes as
$B = \mathfrak P(B) + \mathfrak Q(B)$, yielding the explicit law
\eqref{eq:affine-feedback}, proving (2).

Substituting $u^\star$ into the closed--loop dynamics gives
\[
\dot\xi + \mathfrak{P}(B(\xi,g)) = 0,
\]
which for $\xi-a_0\in\mathfrak d$ becomes precisely 
\eqref{eq:affine-reduced}, proving (3).
\end{proof}

\begin{remark}
We adopt the left--trivialized velocity $\xi=g^{-1}\dot g$, natural for systems
whose constraints and actuation are expressed in body coordinates.
Right--invariant systems, such as the rigid body with an internal rotor,
replace $\xi$ with $\dot g g^{-1}$ and interchange 
$(L_{g^{-1}})_{\ast}$ with $(R_{g^{-1}})_{\ast}$, but all geometric results
remain unchanged.\hfill$\diamond$
\end{remark}


\subsection{Example: Rigid body with an internal rotor}

We consider a rigid body equipped with a single internal rotor, see \cite{bloch1992stabilization}. The configuration space is $G = SO(3) \times \mathbb{S}^{1}$, the Cartesian product of two Lie groups endowed with the product Lie group structure. The $SO(3)$ component describes the attitude of the rigid body, while the $\mathbb{S}^{1}$ component represents the angular position of the internal rotor. Actuation of the rotor redistributes angular momentum internally, without generating external torques.

Let $\mathfrak{g} \cong \mathbb{R}^{4}$ denote the Lie algebra of $G$, and consider body variables $\xi = (\omega, \dot{\alpha}) = (\omega_1,\omega_2,\omega_3,\dot{\alpha}) \in \mathfrak{g}$, where $\omega \in \mathbb{R}^{3}$ is the body angular velocity of the rigid body and $\dot{\alpha}$ is the rotor angular velocity. We endow $\mathfrak{g}$ with the inner product
\[
\langle \xi, \xi \rangle
=
\lambda_{1} \omega_1^2 + \lambda_{2} \omega_2^2 + \lambda_{3} \omega_3^2
+ 2 J \omega_{3} \dot{\alpha}
+ J \dot{\alpha}^{2},
\]
which corresponds to the kinetic energy metric of a rigid body with principal inertias $\lambda_i>0$ and a rotor of inertia $J$ aligned with the third body axis. For the metric to be positive definite we assume $\lambda_i>0$ for $i=1,2,3$ and $J\lambda_{3}-J^{2}>0$.

The adjoint action $\mathrm{ad}_{\xi}:\mathfrak{g}\to\mathfrak{g}$ is given by
\[
\mathrm{ad}_{\xi}\eta = (\omega \times \gamma,0),
\qquad
\eta=(\gamma,\dot{\beta})\in\mathfrak{g}.
\]
The corresponding Euler--Poincar\'e equations describing the free dynamics of the rigid body--rotor system are
\begin{equation*}
\begin{split}
\dot{\omega}_{1} & = -\frac{1}{\lambda_{1}}
\left( (\lambda_{3}-\lambda_{2})\omega_{2}\omega_{3} + J \omega_{2}\dot{\alpha} \right), \\
\dot{\omega}_{2} & = -\frac{1}{\lambda_{2}}
\left( (\lambda_{1}-\lambda_{3})\omega_{1}\omega_{3} - J \omega_{1}\dot{\alpha} \right), \\
\dot{\omega}_{3} & = -\frac{J}{D} (\lambda_{2}-\lambda_{1})\omega_{1}\omega_{2}, \\
\ddot{\alpha} & = \frac{J}{D} (\lambda_{2}-\lambda_{1})\omega_{1}\omega_{2},
\end{split}
\end{equation*}
where $D = J\lambda_{3}-J^{2}$. These equations describe the exchange of angular momentum between the rigid body and the internal rotor while preserving total angular momentum in the absence of control. {The last two equations give us the conservation law $\dot{\omega}_{3}+\ddot{\alpha}=0$.}

Rather than stabilizing the system to rest, we aim to impose a desired rotational regime characterized by a constant affine relation between the body angular velocity $\omega_{3}$ and the rotor speed $\dot{\alpha}$. This is achieved by enforcing the affine nonholonomic constraint
\begin{equation}\label{constraint:rotor}
    (J - k \lambda_{3}) \omega_{3} + J(1-k)\dot{\alpha} = p,
\end{equation}

where $k\in\mathbb{R}$ and $p\in\mathbb{R}$ are design parameters. The constant $p$ defines a prescribed momentum bias, while $k$ interpolates between body-dominated and rotor-dominated angular momentum regimes. 

When $p=0$, the affine constraint reduces to a linear virtual nonholonomic constraint, and the invariant set $\mathfrak{a}$ coincides with its model vector subspace $\mathfrak{h}$. In this case, the closed-loop dynamics are confined to a linear subspace of $\mathfrak{g}$ passing through the origin. For $p\neq 0$, the invariant set is an affine subspace, corresponding to a nonzero momentum bias. 


Let $\{e_1,e_2,e_3,e_4\}$ be the basis of $\mathfrak{g}$ defined by $e_{1}=(\widehat{(1,0,0)},0)$, $e_{2}=(\widehat{(0,1,0)},0)$, $e_{3}=(\widehat{(0,0,1)},0)$, $e_{4}=(\hat{0},1)$, where the hat map is the standard identification between $\mathbb{R}^3$ and $\mathfrak{so}(3)$. The constraint defines an affine subspace $\mathfrak{a} = \xi_0 + \mathfrak{h} \subset \mathfrak{g}$, whose model vector subspace is $\displaystyle{\mathfrak{h}
=
\mathrm{span}
\left\{
e_{1},\;
e_{2},\;
e_{4} - \frac{J(1-k)}{J-k\lambda_{3}}\,e_{3}
\right\}}$. Since $\mathfrak{h}$ is not a Lie subalgebra, the associated constraint is nonholonomic, and the affine offset $\xi_0$ encodes the desired momentum bias.

We choose a complementary input force subspace
\[
\mathfrak{f}
=
\mathrm{span}
\left\{
-\frac{J}{D}e_{3} + \frac{\lambda_{3}}{D}e_{4}
\right\},
\]
so that $\mathfrak{g} = \mathfrak{h} \oplus \mathfrak{f}$. Under actuation along $\mathfrak{f}$, the controlled mechanical system takes the form
\begin{equation*}
\begin{split}
\dot{\omega}_{1} & = -\frac{1}{\lambda_{1}}
\left( (\lambda_{3}-\lambda_{2})\omega_{2}\omega_{3} + J \omega_{2}\dot{\alpha} \right), \\
\dot{\omega}_{2} & = -\frac{1}{\lambda_{2}}
\left( (\lambda_{1}-\lambda_{3})\omega_{1}\omega_{3} - J \omega_{1}\dot{\alpha} \right), \\
\dot{\omega}_{3} & = -\frac{J}{D} (\lambda_{2}-\lambda_{1})\omega_{1}\omega_{2}
-\frac{J}{D}u, \\
\ddot{\alpha} & = \frac{J}{D} (\lambda_{2}-\lambda_{1})\omega_{1}\omega_{2}
+\frac{\lambda_{3}}{D}u .
\end{split}
\end{equation*}
Using the affine virtual nonholonomic constraint construction of Section~V-E, the unique feedback control law rendering the affine constraint invariant is given by
\[
u^{\star} = k(\lambda_{1}-\lambda_{2})\omega_{1}\omega_{2}.
\]
Under this feedback, the closed-loop dynamics evolve on the affine invariant subspace $\mathfrak{a}$, enforcing the desired rotational regime through internal momentum redistribution.

We now present numerical simulations illustrating the closed--loop behavior of the rigid body with an internal rotor under the proposed virtual affine nonholonomic constraint. 
The purpose of these simulations is to validate the geometric structure predicted by the theory, namely the exact invariance of an affine constraint manifold and the resulting regime--shaping effect on the dynamics.

All simulations were performed using the parameter values $(\lambda_1,\lambda_2,\lambda_3) = (2,1,3), \, J = 1$, with constraint parameters $k = 0.33$, $p = 0.5$. The initial conditions were chosen inside the affine constraint manifold defined by \eqref{constraint:rotor}. We set $(\omega_1,\omega_2,\omega_3) = (0.3,\,0.8,\,0.6)$ and $\dot{\alpha}$ is obtained from \eqref{constraint:rotor}.


The feedback control law was applied throughout the simulations. Figure~\ref{omegas} shows the time evolution of the rigid--body angular velocity components $(\omega_1,\omega_2,\omega_3)$. 
Rather than converging to rest, the trajectories evolve in a structured periodic regime with constant amplitude. 

\begin{figure}[htb!]
    \centering
    \includegraphics[width=0.8\linewidth]{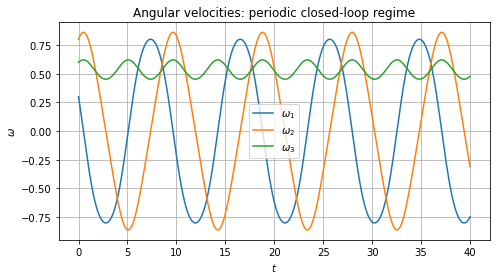}
    \caption{Time evolution of the rigid--body angular velocity components $(\omega_1,\omega_2,\omega_3)$ under the virtual affine nonholonomic constraint. The closed--loop system exhibits a bounded periodic regime rather than convergence to equilibrium.}
    \label{omegas}
\end{figure}

The exact invariance of the affine constraint manifold is illustrated in Fig.~\ref{constraints}, where the affine constraint function
\[
\Phi(t) = (J-k\lambda_3)\omega_3(t) + J(1-k)\dot{\alpha}(t) - p
\]
remains constant and equal to zero along the entire trajectory. 
\begin{figure}[htb!]
    \centering
    \includegraphics[width=0.8\linewidth]{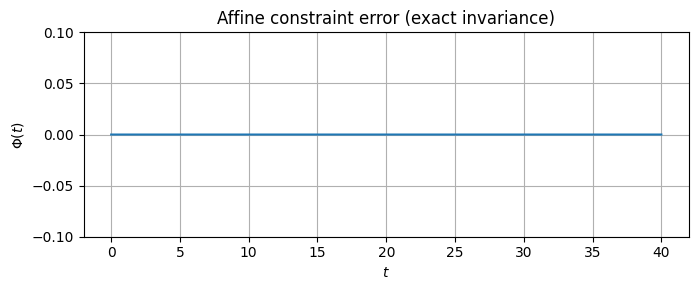}
    \caption{Evolution of the affine constraint function $\Phi(t)$. The constraint function vanishes, demonstrating exact invariance of the affine constraint manifold under the closed--loop dynamics.}
    \label{constraints}
\end{figure}

The geometric structure of the closed--loop dynamics is further highlighted in Fig.~\ref{fig:energy_sphere}, which depicts the trajectory of $(\omega_1,\omega_2,\omega_3)$ on the angular--velocity phase space.

Under the virtual affine constraint, the motion is confined to a closed curve corresponding to the intersection of this energy surface with the affine invariant manifold. The trajectory evolves along a closed curve, consistent with the periodic regime observed in Fig.~\ref{omegas}. 
This confirms that the imposed affine constraint shapes the long--term behavior of the system without introducing dissipation or convergence to equilibrium.


\begin{figure}[htb!]
    \centering
    \includegraphics[width=0.8\linewidth]{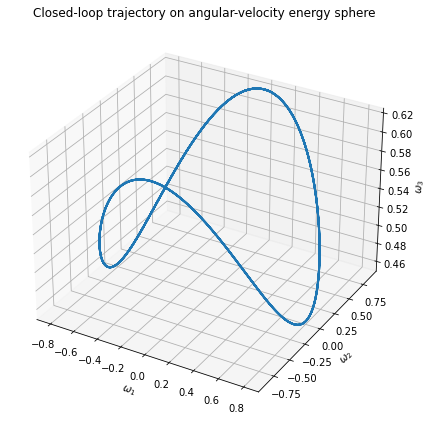}
    \caption{Closed--loop trajectory of the angular velocity $(\omega_1,\omega_2,\omega_3)$ in angular--velocity space. The motion evolves along a closed curve, consistent with the periodic regime induced by the virtual affine constraint.}
    \label{fig:energy_sphere}
\end{figure}

For completeness, Fig.~\ref{fig:projection} shows the projection of the closed--loop trajectory onto the $(\omega_1,\omega_2)$ plane. 
This planar representation provides an intuitive visualization of the periodic motion induced by the virtual affine constraint and highlights the bounded nature of the closed--loop response.

\begin{figure}[htb!]
    \centering
    \includegraphics[width=0.8\linewidth]{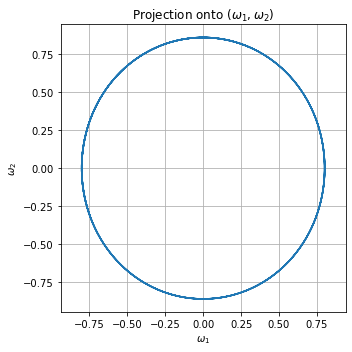}
    \caption{Projection of the closed--loop trajectory onto the $(\omega_1,\omega_2)$ plane, illustrating the periodic rotational regime enforced by the affine constraint.}
    \label{fig:projection}
\end{figure}

\section*{Conclusions}

This paper developed a unified geometric framework for virtual holonomic and virtual nonholonomic constraints on Lie groups, formulated intrinsically through feedback--invariant manifolds of the Lie algebra. Both linear and affine constraints were treated within a single setting. For each class of constraint, we established existence and uniqueness of enforcing feedback laws under the transversality condition $\mathfrak{g} = \mathfrak{h} \oplus \mathfrak{f}$, and showed that the resulting closed--loop trajectories evolve as geodesics of an induced constrained connection.

In the holonomic case, virtual constraints reproduce the reduced Euler--Lagrange equations on the constraint submanifold, and when the constraint distribution corresponds to a Lie subalgebra, the motion reduces to the dynamics on a Lie subgroup. In the nonholonomic case, we identified precise geometric conditions under which virtual nonholonomic constraints reproduce the reduced dynamics of classical nonholonomic systems. The affine extension demonstrates that the same geometric mechanisms apply to drift--affine constraints, such as those arising in rigid bodies with internal rotors, leading to invariant affine manifolds and nontrivial closed--loop regimes.

The proposed framework applies directly to robotic and mechanical systems evolving on Lie groups such as $SO(3)$ and $SE(3)$, including quadrotor UAVs and rigid bodies with internal actuation. The examples illustrate how virtual constraints can be systematically used to encode low-dimensional motion primitives and shape qualitative dynamical behavior, while preserving the intrinsic geometric structure of the configuration space.

\bibliographystyle{siamplain}
\bibliography{references}

\end{document}